\documentclass[3p,preprint]{elsarticle}
\usepackage{graphicx}
\usepackage{amssymb}
\usepackage{amsthm}
\usepackage{amsmath}
\usepackage{revsymb}
\usepackage{empheq}
\usepackage{mathtools}
\usepackage{subfigure}
\usepackage{lineno} 
\usepackage{enumerate}
\usepackage{multirow} 
\usepackage{caption} 
\usepackage{array} 
\usepackage{booktabs} 
\usepackage{tabularx}
\usepackage{longtable}
\usepackage{bm}
\usepackage{setspace}
\usepackage{float}
\usepackage{bbm}
\usepackage[hidelinks]{hyperref}
\usepackage{algorithm}
\usepackage{algorithmic}
\journal{arXiv}
\usepackage[figuresright]{rotating}

\newcounter{thmcounter}[section]
\theoremstyle{definition}
\newtheorem{definition}[thmcounter]{Definition}
\newtheorem{example}[thmcounter]{Example}

\newtheorem{remark}[thmcounter]{Remark}
\theoremstyle{plain}
\newtheorem{theorem}[thmcounter]{Theorem}

\newtheorem{lemma}[thmcounter]{Lemma}
\newtheorem{corollary}[thmcounter]{Corollary}

\numberwithin{equation}{section}
\numberwithin{thmcounter}{section}

\allowdisplaybreaks[3]

\begin{document}

\begin{frontmatter}

\title{A class of thin-film equations with space-time dependent 
gradient nonlinearity and its application to image sharpening}

\author{Yuhang Li$^{\ \textrm{a}}$}
\ead{mathlyh@stu.hit.edu.cn}
\author{Zhichang Guo$^{\ \textrm{a},}\cormark[1]$}
\ead{mathgzc@hit.edu.cn}
\author{Fanghui Song$^{\ \textrm{a}}$}
\ead{21b912024@stu.hit.edu.cn}
\author{Boying Wu$^{\ \textrm{a}}$}
\ead{mathwby@hit.edu.cn}
\author{Bin Guo$^{\ \textrm{b}}$}
\ead{bguo@jlu.edu.cn}
\cortext[1]{Corresponding author.}

\address{$^{\rm{a}}$The Department of Mathematics, Harbin Institute of Technology, Harbin 150001, China}
\address{$^{\rm{b}}$School of Mathematics, Jilin University, Changchun, Jilin Province 130012, China}

\begin{abstract}
    We introduce a class of thin-film equations with space-time dependent 
    gradient nonlinearity and apply them to image sharpening. By modifying the potential well 
    method, we overcome the challenges arising from variable exponents and the ``moving" Nehari 
    manifold, thereby providing a classification of global existence and finite time blow-up of 
    solutions under different initial conditions. Decay rates for global solutions and 
    lifespan estimates for blow-up solutions are also established. 
    Numerical experiments demonstrate the effectiveness of this class of equations as 
    image sharpening models.
\end{abstract}

\begin{keyword}
thin-film equation \sep time-dependent nonlinearity \sep variable exponent 
\sep global existence and blow-up \sep image processing
\end{keyword}

\end{frontmatter}

\section{Introduction} \label{sec1}

Let $\Omega \subset \mathbb{R}^N$ ($N \geq 2$) be a bounded domain. 
For $\alpha \in \mathbb{R}$, we consider the following 
initial-boundary value problem:

\begin{equation}
    \left\{
    \begin{array}{ll}
        u_t + \mathcal{L}_{\alpha} u =
        \operatorname{div} \left(F(t, x, \nabla u)\right), 
         & x \in \Omega, \ t > 0,  \\ 
        u(0,x) = u_0(x), & x \in \Omega,  \\
        u = 0, & x \in \mathbb{R}^N \backslash \Omega.
    \end{array}
    \right. \label{eqn:main}
\end{equation}
Here, $\mathcal{L}_{\alpha} = (-\Delta)^2 + \alpha (-\Delta)^{2s}$ is a 
mixed local-nonlocal biharmonic operator, $s \in (0,1)$,
$F(t, x, \xi) = \xi - k(t)|\xi|^{p(x)-2}\xi$, $p: \mathbb{R}^N \to (2,+\infty)$.
This type of problem appears in the modeling of epitaxial 
thin film growth. On a square domain, letting $u(t,x)$ denote the height of the thin film, 
the continuum model for epitaxial thin film growth is given by \cite{ZANGWILL19968}:
\begin{equation}
    u_t = - \operatorname{div} (\mathbf{j}) + g + \eta, \label{eqn:mb-law}
\end{equation}    
where $\mathbf{j}=\mathbf{j}(t,x)$, $g=g(t,x)$, $\eta=\eta(t,x)$ denotes 
all process of moving atoms along the surface,
the deposition flux and Gaussian noise, respectively. 
One of the most widespread treatments is to neglect the noise term and assume 
\cite{ZANGWILL19968,ORTIZ1999697}
\[\mathbf{j}=A_1 \nabla u + A_2 \nabla(\Delta u) + A_3 |\nabla u|^{p-2} \nabla u,\]
so that \eqref{eqn:mb-law} reduces to
\begin{equation}
u_t + A_1 \Delta u + A_2 \Delta^2 u + A_3 
\operatorname{div} \left(|\nabla u|^{p-2} \nabla u\right) = g. \label{eqn:mbe}
\end{equation} 
In this case, the term $A_1 \Delta u$ corresponds to diffusion driven 
by evaporation-condensation, $A_2 \Delta^2 u$ represents capillarity-driven 
surface diffusion, and $A_3 \operatorname{div} \left(|\nabla u|^{p-2} \nabla u\right)$ 
accounts for the upward hopping of atoms \cite{MULLINS1957333,DASSARMA19923762}.

The intriguing mathematical properties of \eqref{eqn:mbe} have attracted extensive 
research interest over the past few decades. King et al. \cite{KING2003459} 
studied the global existence and asymptotic behavior of solutions to \eqref{eqn:mbe} 
in the case where $A_1>0, A_2>0, A_3<0$ and $g=0$. Later, Kohn et al. \cite{KOHN20031549}
analyzed the coarsening behavior of solutions under the same setting. 
Sandjo et al. \cite{SANDJO20157260} investigated the local well-posedness in $L^p$ 
space of \eqref{eqn:mbe} when $A_1=0, A_2>0, A_3 \ne 0$; afterward, Li et al. \cite{LI202425}
established well-posedness results in local BMO space for $p=4$ and further studied 
the stability for global solutions for small initial data in VMO space. Related problems 
include variants of \eqref{eqn:mbe} involving nonlinear source terms 
\cite{LI201696,HAN2018451,LI20239147}, time-dependent 
coefficients \cite{PHILIPPIN20152507,PHILIPPIN2016718,ZHOU2021128}, 
or Hessians \cite{ESCUDERO2015924,MARRAS2024109253}. 

In recent years, particular attention 
has been devoted to the case where $A_2>0$ and $A_3>0$. Ishige \cite{ISHIGE2020927} 
et al. studied the local existence and sufficient conditions for blow-up of solutions 
to the Cauchy problem for
\begin{equation}
    u_t + (-\Delta)^2 u = -\operatorname{div} 
    \left(|\nabla u|^{p-2}\nabla u\right), \label{eqn:a10a2pa3p}
\end{equation}
on $\mathbb{R}^N$. Feng et al. \cite{FENG2022561} investigated the global existence and 
blow-up for \eqref{eqn:a10a2pa3p} with and without advection on the 
torus $\mathbb{T}^2$. Miyake et al. \cite{MIYAKE2021247} first established the asymptotic 
behavior of global solutions with subcritical initial energy under Neumann boundary 
conditions, and Zhao et al. \cite{ZHAO202412} subsequently completed the classification of 
global existence and blow-up for this problem. The problem \eqref{eqn:main} studied in this 
paper corresponds to the case $A_1<0$, $A_2>0$ and $A_3>0$. Specifically, by 
setting $\alpha=0$ and taking $F(t,x,\xi) = -|\xi|^{p-2}\xi$ in \eqref{eqn:main}, we 
recover the form of \eqref{eqn:a10a2pa3p}.

We claim that \eqref{eqn:main} can serve as an image sharpening model.
Image sharpening refers to the enhancement of edges in an image to make it
clearer and more visually appealing. Unlike image restoration, which 
emphasizes fidelity to a clean reference image, sharpening focuses more 
on improving the overall visual clarity \cite{GONZALEZ2008}. 
It is commonly used in scenarios where images are intended for public 
display and aesthetic presentation. The simplest sharpening PDE  
is the linear backward diffusion equation, which is ill-posed and 
extremely sensitive to noise in the image, and thus requiring 
some form of stabilization or regularization techniques in practical computations 
\cite{OSHER1991414,TERNAT2011266,XIONG201481}. 
Aside from linear backward diffusion, other well-known sharpening PDEs include 
nonlinear backward diffusion equations \cite{POLLAK2000256,BERGERHOFF2020,SCHAEFER20221},
anisotropic forward-backward diffusion equations 
\cite{GILBOA2002689,GILBOA2004121,CALDER20111536}, and a class of nonlinear hyperbolic 
equations known as shock filters \cite{OSHER1990919}.
 
Problem \eqref{eqn:main} involves a mixed local-nonlocal biharmonic operator and a 
space-time dependent nonlinear gradient term. Early works 
\cite{YOU20001723,LYSAKER20031579,BERTOZZI2004764,HAJIABOLI2011177,LI20132117} 
employed fourth-order PDEs for image restoration, 
demonstrating their effectiveness in avoiding staircase effects and preserving 
ramps (affine regions). A recent study \cite{SONG2024926} applied a thin film-type 
equation to image segmentation and 
reported favorable results. As nonlocal operators have been found to be effective 
in texture preservation \cite{GILBOA20091005,CHAN2013276,SUN2014691,YAO2019839,WEN2023453}, 
mixed local and nonlocal operators have recently attracted 
increasing attention in image processing  
\cite{CHEN20197611,DELON2019458,SHI2021103362,SHI2023686}. Moreover, spatially or 
space-time dependent differential operators are commonly used in image processing models as 
they can more effectively exploit the intrinsic features of the image and  
improve the self-adaptability of models 
\cite{CHEN20061383,GUO2012958,TIIROLA201756,SHAO2020103166,YANG2023108627}.

Inspired by the above works and ideas, we propose \eqref{eqn:main} as an image sharpening model, 
which can be analyzed as follows:
if we disregard the term $\mathcal{L}_{\alpha} u$ on the left-hand side
and focus on the right-hand side term 
$\operatorname{div}\left((1-k(t)|\nabla u|^{p(x)-2})\nabla u\right)$, one can 
identify a space-time dependent threshold function 
$\mathfrak{T}(t,x)=(k(t))^{\frac{1}{2-p(x)}}$, such that:
when $|\nabla u|<\mathfrak{T}(t,x)$, the diffusion coefficient $1-k(t)|\nabla u|^{p(x)-2}$ 
is positive, exhibiting forward diffusion behavior; 
when $|\nabla u|>\mathfrak{T}(t,x)$, the coefficient becomes negative, 
leading to backward diffusion. 
By appropriately designing the functions $k(t)$ and $p(x)$, one can control 
the evolution process and ensure that backward diffusion occurs primarily near image edges, 
thereby achieving edge enhancement. 
The term $\mathcal{L}_{\alpha} u$, on the other hand, acts as a viscosity-type term 
that helps prevent the model from becoming ill-posed and reduces the influence of noise during 
the sharpening process. By tuning the parameter $\alpha$, one can achieve edge enhancement 
without amplifying noise in smooth regions, and maintaining fine details such as ramps and 
textures.

The main contribution of this paper is the theoretical analysis of problem \eqref{eqn:main}, 
especially the classification of global existence and finite time blow-up, which can be 
regarded as an extension and continuation of previous works \cite{MIYAKE2021247,ZHAO202412} 
on \eqref{eqn:a10a2pa3p}. The analyses in \cite{MIYAKE2021247} and \cite{ZHAO202412} are 
mainly based on the potential well method, originally introduced by Sattinger and Payne   
\cite{SATTINGER1968148,PAYNE1975273} and modified by Liu et al. \cite{LIU20062665}, 
Gazzola et al. \cite{GAZZOLA2005961} and Xu et al. \cite{XU20132732}. Besides its 
application to thin-film equations \cite{ESCUDERO2015924,LI201696,DONG201789,HAN2018451,
MIYAKE2021247,ZHAO202412}, the potential well method has also been 
utilized in the study of other types of equations, such as reaction-diffusion equations 
\cite{KBIRIALAOUI20141723,XIANG20183328,DING20201046,GUO202245},
pseudo-parabolic equations \cite{CHEN20154424,SUN20183685,CHEN2024395,YUAN2025346}, and wave 
equations \cite{GAZZOLA2006185,YANG2022569,ZHANG2024065011,DING2025055012}. 
For problem \eqref{eqn:main}, the application of the potential well method faces two main 
difficulties. First, the time-dependent coefficient $k(t)$ result in a ``moving" Nehari 
manifold, which brings challenges to the verification of invariant sets and the analysis 
for high initial energy cases. Second, the presence of variable exponents makes an 
explicit expression for the potential well depth $d(t)$ cannot generally be obtained, 
and additional technical challenges are introduced in the derivation of differential 
inequalities. In this paper, we incorporate techniques from 
\cite{KBIRIALAOUI20141723,SUN20183685,GUO202245}, further modify the classical potential well 
method to overcome these difficulties and combine it with the Galerkin method, the differential 
inequality technique, and Levine’s concavity argument \cite{LEVINE1973371} to classify 
the global existence and finite time blow-up of solutions. Moreover, 
we investigate decay rates for global solutions as well as 
lifespan estimates for blow-up solutions.

The rest of this paper is organized as follows. In Sect. \ref{sec2}, we introduce necessary 
notations, definitions, assumptions and establish a series of preliminary results for potential 
wells that will be used later. Section \ref{sec3} establishes the local existence of 
solutions to \eqref{eqn:main}. In Sect. \ref{sec4}, we provide threshold results for the 
solutions to exist globally or to blow up in finite time when $J(u_0; 0) \leq \underline{d}$, 
and get a decay rate of the $H^2$-norm for global solutions.
Section \ref{sec5} gives sufficient conditions for global existence 
when $J(u_0; 0) \geq \overline{d}$ and for finite time blow-up when $J(u_0; 0) > 0$
(the notations $J(u_0; 0)$, $\underline{d}$ and $\overline{d}$ will be defined in Sect. \ref{sec2}). 
A lower bound estimate for the blow-up time is established in Sect. \ref{sec6}.
Finally, in Sect. \ref{sec7}, we propose a numerical scheme for solving \eqref{eqn:main} and 
present several numerical examples, where experiments on images are conducted to verify 
the effectiveness of \eqref{eqn:main} as an image sharpening model.

\section{Notations, assumptions and preliminary results} \label{sec2}

\subsection{Function space setting and the fractional Laplacian}

In this paper, we always let $\Omega \subset \mathbb{R}^N$ be a 
bounded domain with appropriately smooth boundary $\partial \Omega$. 
For $1 \leq p < + \infty$, we set a closed subspace of $L^p(\mathbb{R}^N)$ by
\[L_0^p(\Omega) \coloneqq \left\{v \in L^p(\mathbb{R}^N) : v=0 ~~\textrm{a.e.}
~\textrm{in}~~ \mathbb{R}^N \backslash \Omega \right\},\]
which can be identified with $L^p(\Omega)$ by zero extension outside $\Omega$. 
In particular, we set $\mathcal{H}_0(\Omega) := L_0^2(\Omega)$. 
From now on, $L_0^p(\Omega)$ will be 
simply denoted by $L^p(\Omega)$ and the same notation for functions in $L^p(\Omega)$  
and $L_0^p(\Omega)$ will be used if there is no confusion.
We denote by $\|\cdot\|_p$ the $L^p(\mathbb{R}^N)$-norm and $(\cdot,\cdot)$ 
the $L^2(\mathbb{R}^N)$-inner product. Furthermore, for $k \in \mathbb{N}$, we set 
a closed subspace of $W^{k,p}(\mathbb{R}^N)$ by
\[\mathcal{W}_0^{k,p}(\Omega) \coloneqq \left\{v \in W^{k,p}(\mathbb{R}^N) : 
v=0~~\textrm{a.e.}~\textrm{in}~~ \mathbb{R}^N \backslash \Omega \right\}.\]
The regularity assumption on $\partial \Omega$ implies that we can identify 
the space $\mathcal{W}_0^{k,p}(\Omega)$ with $W_0^{k,p}(\Omega)$ in the following sense: 
a function $v \in \mathcal{W}_0^{k,p}(\Omega)$ can be viewed as the zero
extension of its restriction to $\Omega$, 
with $v|_{\Omega} \in W_0^{k,p}(\Omega)$. Indeed, 
$\mathcal{W}_0^{k,p}(\Omega)$ is the completion of $C_0^{\infty}(\Omega)$ with respect to 
the $W^{k,p}(\mathbb{R}^N)$-norm. To simplify the notation, we here and henceforth 
continue to use $C_0^{\infty}(\Omega)$ to denote the set of all 
functions $\varphi: \mathbb{R}^N \to \mathbb{R}$ such that 
$\operatorname{supp} \varphi \subset \Omega$ and 
$\varphi|_{\Omega} \in C_0^{\infty}(\Omega)$. In particular, 
we set $\mathcal{H}_0^2(\Omega) := \mathcal{W}_0^{2,2}(\Omega)$. 

We will need the following Tartar's type inequality.
\begin{lemma}[\cite{DIAZ19941085}]
    Let $d \in \mathbb{N}$ and $p \in [2,\infty)$. For any $a, b \in \mathbb{R}^d$, 
    it holds that 
    \begin{equation} \label{eqn:tartar-inequality}
        \left||a|^{p-2}a-|b|^{p-2}b\right| \leq (p-1)^{2-\frac{2}{p}} |a-b|
     (|a|^p+|b|^p)^{\frac{p-2}{p}}.
    \end{equation}
\end{lemma}

Let $O \subset \mathbb{R}^N$ be an open set. Denote by $\mathcal{P}(O)$ 
the set of all measurable functions $p : O \to [1,\infty]$. We also denote 
$p^- = \operatorname*{ess\,inf}_{x \in \Omega} p(x)$ and 
$p^+ = \operatorname*{ess\,sup}_{x \in \Omega} p(x)$. 
For $p \in \mathcal{P}(O)$, we set a semimodular
\[\varrho(v) =\varrho_{L^{p(\cdot)}(O)}(v) \coloneqq 
\int_{O \backslash O_{\infty}} |v(x)|^{p(x)} dx 
+ \operatorname*{ess\,sup}_{x \in O_{\infty}} |v(x)|,\]
where $O_{\infty} = \{x \in O : p(x) = \infty\}$. The variable exponent Lebesgue 
space $L^{p(\cdot)}(O)$ is
\[L^{p(\cdot)}(O) \coloneqq \left\{v \textrm{ is measurable } : 
\varrho_{L^{p(\cdot)}(O)}(v) < + \infty \right\}\]
equipped with the norm 
\[\|v\|_{p(\cdot)} =\|v\|_{L^{p(\cdot)}(O)} \coloneqq \inf \left\{ 
    \lambda > 0 : \varrho_{L^{p(\cdot)}(O)}\left(\frac{v}{\lambda}\right) \leq 1
\right\}.\]
The variable exponent Sobolev space is defined by 
\[W^{1,p(\cdot)}(O) \coloneqq \left\{v \in L^{p(\cdot)}(O) : 
|\nabla v|^{p(x)} \in L^{1}(O) \right\},\]
which is endowed with the norm 
\[\|v\|_{W^{1,p(\cdot)}(O)} \coloneqq \|v\|_{p(\cdot)} 
+ \|\nabla v\|_{p(\cdot)}.\]
For $p \in \mathcal{P}(O)$, the spaces $L^{p(\cdot)}(O)$ and $W^{1,p(\cdot)}(O)$ are 
Banach spaces, which are separable if $p$ is bounded, and reflexive and uniformly 
convex if $1 < p^- \leq p^+ < \infty$. A measurable function $a: O \to \mathbb{R}$ is 
said to be globally log-H{\"o}lder continuous in $O$ iff there exist constants 
$A_1>0$, $A_2>0$ and $a_{\infty} \in \mathbb{R}$ such that 
\[|a(x)-a(y)| \leq \frac{A_1}{\log \left(e + \frac{1}{|x-y|}\right)} \quad 
\textrm{and} \quad |a(x)-a_{\infty}| \leq \frac{A_2}{\log \left(e + |x|\right)}\]
for all $x, y \in O$. Denote by $\mathcal{P}^{\log}(O)$ 
the set of all functions $p \in \mathcal{P}(O)$ such that
$\frac{1}{p}$ is globally log-H{\"o}lder continuous in $O$. 
Let $p \in \mathcal{P}^{\log}(\Omega)$ be bounded, we denote by 
$W_0^{1,p(\cdot)}(\Omega)$ the closure of $C_0^{\infty}(\Omega)$ in 
the space $W^{1,p(\cdot)}(\Omega)$. The space $W_0^{1,p(\cdot)}(\Omega)$ is 
separable, and if $1 < p^- \leq p^+ < \infty$, it is also 
reflexive and uniformly convex. For bounded $p \in \mathcal{P}^{\log}(\mathbb{R}^N)$, 
we define a closed subspace of $W^{1,p(\cdot)}(\mathbb{R}^N)$ by
\[\mathcal{W}_0^{1,p(\cdot)}(\Omega) \coloneqq 
\overline{C_0^{\infty}(\Omega)}^{W^{1,p(\cdot)}(\mathbb{R}^N)}
= \left\{v \in W^{1,p(\cdot)}(\mathbb{R}^N) : 
v=0~~\textrm{a.e.}~\textrm{in}~~ \mathbb{R}^N \backslash \Omega \right\}.\]
Indeed, each function $v \in \mathcal{W}_0^{1,p(\cdot)}(\Omega)$ can be viewed as 
the zero extension of its restriction to $\Omega$ with 
$v|_{\Omega} \in W_0^{1,p(\cdot)}(\Omega)$, so we can identify 
$\mathcal{W}_0^{1,p(\cdot)}(\Omega)$ with $W_0^{1,p(\cdot)}(\Omega)$. For more 
properties of variable exponent Lebesgue and Sobolev spaces, see 
\cite{DIENING2011,RADULESCU2015,ANTONTSEV2015}. Here, we outline 
several relevant properties that will be applied in subsequent sections.

\begin{lemma}[\cite{DIENING2011,RADULESCU2015}]
    Let $p, q \in \mathcal{P}(\Omega)$ be bounded, $v \in L^{p(\cdot)}(\Omega)$. Denote 
    $\varrho =\varrho_{L^{p(\cdot)}(\Omega)}$, then 
    \begin{enumerate}[(i)]
        \item $\|v\|_{p(\cdot)}<1$ (respectively, $\|v\|_{p(\cdot)}=1$, 
        $\|v\|_{p(\cdot)}>1$) iff 
        $\varrho(v)<1$ (respectively, $\varrho(v)=1$, $\varrho(v)>1$).
        \item If $\|v\|_{p(\cdot)}<1$, then 
        $\|v\|_{p(\cdot)}^{p^+} \leq \varrho(v) \leq \|v\|_{p(\cdot)}^{p^-}$; 
        if $\|v\|_{p(\cdot)}>1$, then 
        $\|v\|_{p(\cdot)}^{p^-} \leq \varrho(v) \leq \|v\|_{p(\cdot)}^{p^+}$.
        \item If $p(x) \leq q(x)$ a.e. in $\Omega$, then $L^{q(\cdot)}(\Omega)$ 
        is continuous embedded in $L^{p(\cdot)}(\Omega)$. 
    \end{enumerate}
\end{lemma}

\begin{lemma}[\cite{DIENING2011,RADULESCU2015}]
    Let $p \in \mathcal{P}^{\log}(\mathbb{R}^N)$ and $q \in \mathcal{P}(\Omega)$ 
    be bounded, satisfying $q(x) < p^*(x)$ for a.e. $x \in \overline{\Omega}$,  
    where $p^*(x) = \infty$ if $N \leq p(x)$ and $p^*(x) = \frac{Np(x)}{N-p(x)}$ if $N>p(x)$. 
    Then $W_0^{1,p(\cdot)}(\Omega)$ is compactly embedded into $L^{q(\cdot)}(\Omega)$.
\end{lemma}

\begin{lemma}[H{\"o}lder's inequality \cite{DIENING2011}]
    Let $p, q ,r \in \mathcal{P}(O)$ be such that 
    $\frac{1}{r(x)} = \frac{1}{p(x)}+\frac{1}{q(x)}$ for almost every $x \in O$. 
    Then for all $f \in L^{p(\cdot)}(O)$ and $g \in L^{q(\cdot)}(O)$,
    \[\|fg\|_{r(x)} \leq 2 \|f\|_{p(\cdot)} \|g\|_{q(\cdot)}.\]
\end{lemma}

Using H{\"o}lder's inequality, we prove a time continuity result in the space 
$C([0,T];\mathcal{W}_0^{1,p(\cdot)}(\Omega))$.

\begin{lemma} \label{lem:modular-continuity}
    Let $T \in (0,+\infty)$, $p \in \mathcal{P}^{\log}(\mathbb{R}^N)$ be bounded 
    satisfying $p^- > 1$. 
    Let $u(t) \in C([0,T];\mathcal{W}_0^{1,p(\cdot)}(\Omega))$. Then 
    $\varrho_{\kappa}(|\nabla u(t)|) \in C([0,T])$ for any bounded 
    $\kappa: \Omega \to \mathbb{R}_{>0}$. Here 
    \[\varrho_{\kappa}(v) \coloneqq \int_{\Omega} \kappa(x) |v(x)|^{p(x)} dx.\]
\end{lemma}

\begin{proof}
    For $t,t' \in [0,T]$, since $u(t) \in C([0,T];\mathcal{W}_0^{1,p(\cdot)}(\Omega))$, 
    it follows that as $t' \to t$, 
    $\left\|\nabla u(t') - \nabla u(t)\right\|_{p(\cdot)} \to 0$. 
    Set $q(x)=\frac{p(x)}{p(x)-1}$.
    Using H{\"o}lder's inequality and the fact that $p$ and $\kappa$
    are bounded, we have
    \begin{align*} 
        & \left|\varrho_{\kappa}(|\nabla u(t')|) - \varrho_{\kappa}(|\nabla u(t)|)\right| \\
        & \lesssim \int_{\Omega} 
        \left| |\nabla u(t')|^{p(x)} - |\nabla u(t)|^{p(x)}\right| dx \\
        & \lesssim \int_{\Omega} \left(|\nabla u(t')|^{p(x)-1} 
        + |\nabla u(t)|^{p(x)-1}\right) \left|\nabla u(t') - \nabla u(t)\right| dx \\
        & \lesssim \left\|
        |\nabla u(t')|^{p(x)-1}+|\nabla u(t)|^{p(x)-1}\right\|_{q(\cdot)} 
        \left\|\nabla u(t') - \nabla u(t)\right\|_{p(\cdot)} \\ 
        & \lesssim \max \left\{ \left(\varrho(|\nabla u(t')|) + \varrho(|\nabla u(t)|)
        \right)^{\frac{p^--1}{p^-}}, 
        \left(\varrho(|\nabla u(t')|) + \varrho(|\nabla u(t)|)\right)^{\frac{p^+-1}{p^+}}\right\} 
        \left\|\nabla u(t') - \nabla u(t)\right\|_{p(\cdot)} \to 0
    \end{align*}
    as $t' \to t$, which concludes the proof.
\end{proof}

We denote the Fourier transform of $v$ by $\widehat{v}$ and its 
inverse by $\mathcal{F}^{-1} v$. For $r \in \mathbb{R}$, the Bessel 
potential space $H^r(\mathbb{R}^N)$ is defined by 
\[H^r(\mathbb{R}^N) = \left\{v \in L^2(\mathbb{R}^N) :
\|v\|_{H^r(\mathbb{R}^N)} < +\infty \right\},\]
where
\[\|v\|_{H^r(\mathbb{R}^N)} = \left\|\mathcal{F}^{-1} \left( 
    \left(1+|\xi|^2\right)^{\frac{r}{2}} \widehat{v}
\right)\right\|_{L^2(\mathbb{R}^N)}.\]
For $\sigma \in \mathbb{R}$, a bounded linear operator 
$(- \Delta)^r: H^{\sigma}(\mathbb{R}^N) \to H^{\sigma-2r}(\mathbb{R}^N)$ is 
defined by  
\[(- \Delta)^r v = \mathcal{F}^{-1} \left(|\xi|^{2r} \widehat{v} \right),\]
which is known as the fractional Laplacian. In particular, for $s \in (0,1)$, 
$(- \Delta)^{2s} : H^{2s}(\mathbb{R}^N) \to H^{-2s}(\mathbb{R}^N)$ can be 
given weakly by 
\[\left\langle (- \Delta)^{2s} v, \varphi \right\rangle
_{H^{-2s}(\mathbb{R}^N) \times H^{2s}(\mathbb{R}^N)} 
\coloneqq \int_{\mathbb{R}^N} (- \Delta)^{s} v (- \Delta)^{s} \varphi dx\]
for all $v, \varphi \in H^{2s}(\mathbb{R}^N)$. By Plancherel's theorem, it 
is easy to see that $H^{2}(\mathbb{R}^N)$ is continuously embedded into 
$H^{2s}(\mathbb{R}^N)$ and 
\[\|v\|_{H^{2s}(\mathbb{R}^N)}^2 = \|v\|_{L^{2}(\mathbb{R}^N)}^2 
+ \|(- \Delta)^{s} v\|_{L^{2}(\mathbb{R}^N)}^2.\]

Before introducing the potential well, we provide the definition of weak solutions 
to problem \eqref{eqn:main}. 
\begin{definition}[Weak solution] \label{def:main-weak}
    Given $u_0 \in H_0^2(\Omega)$, we say that a function 
    $u = u(t,x) \in C([0,T];\mathcal{H}_0(\Omega))$ is a weak solution to 
    problem \eqref{eqn:main} on $[0,T]$ if it satisfies:

    \begin{enumerate}[(i)]
        \item $u \in L^{\infty}(0,T;\mathcal{H}_0^2(\Omega)) \cap H^1(0,T;\mathcal{H}_0(\Omega))$; 
        \item $u(0,\cdot)=u_0$; \label{enm:main-weak-2}
        \item For any $\varphi \in \mathcal{H}_0^2(\Omega)$ and almost all $t \in (0,T)$, 
        the following identity holds:
        \begin{equation}
            (u_t, \varphi) + (\Delta u, \Delta \varphi) +
            \alpha ((-\Delta)^{s} u, (-\Delta)^{s} \varphi) + 
            \int_{\Omega} F(t, x, \nabla u) \cdot \nabla \varphi d x
            = 0. 
            \label{eqn:main-weak}
        \end{equation}
    \end{enumerate}
\end{definition}

\noindent To make the last term on the left side of \eqref{eqn:main-weak} meaningful, it is necessary 
to impose a restriction on $p$, which will be described in the next subsection.

\subsection{The potential well framework}

Let us start with making the following assumptions on $p, \alpha$ and the function $k(t)$:
\begin{enumerate}
    \item[(H1)] $p \in \mathcal{P}^{\log}(\mathbb{R}^N)$ is bounded and $2<p^- \leq p^+<2^*$. 
    Here $2^*=\infty$ for $N=2$ and $2^*=\frac{2N}{N-2}$ for $N \geq 3$.   
    \item[(H2)] Either $\alpha > -\lambda_1^{1-s}(1-s)^{1-s}s^{-s}$, or 
    $\alpha > -(2-2s)^{2s-2}(2s-1)^{1-2s}$ with $s \in \left(\frac{1}{2},1\right)$. 
    Here $\lambda_1$ denotes the principle eigenvalue of the Dirichlet problem 
    \[-\Delta v = \lambda v ~~\textrm{in}~~ \Omega, \quad v = 0 
    ~~\textrm{on}~~ \partial \Omega.\]
    \item[(H3)] $k \in C^1([0,+\infty))$, $k(0)>0$ and for every $t \geq 0$, $k'(t) \geq 0$.  
\end{enumerate}

For $v \in \mathcal{H}_0^2(\Omega)$ and 
any $t \geq 0$, the time-dependent energy functional and Nehari functional associated with 
problem \eqref{eqn:main} are given by
\[J(v;t) = \frac{1}{2}\|\nabla v\|_2^2 + \frac{1}{2} \|\Delta v\|_2^2 
+ \frac{\alpha}{2}  \|(-\Delta)^{s} v\|_2^2 - k(t)
\int_{\Omega} \frac{1}{p(x)}|\nabla v|^{p(x)} dx,\]
and
\[I(v;t) = \|\nabla v\|_2^2 +  \|\Delta v\|_2^2 
+ \alpha  \|(-\Delta)^{s} v\|_2^2 - k(t) \int_{\Omega} |\nabla v|^{p(x)} dx.\]
(H1) ensures that both $J(v;t)$ and $I(v;t)$ are well-defined. It is not difficult to verify 
that for each $t \geq 0$, both $J(\cdot;t)$ and $I(\cdot;t)$ are continuous 
on $\mathcal{H}_0^2(\Omega)$. (H2) enables us to obtain the following result concerning the equivalence 
of norms on $\mathcal{H}_0^2(\Omega)$.

\begin{lemma}
    Let (H2) be fulfilled, $v \in \mathcal{H}_0^2(\Omega)$. Then 
    \begin{equation} \label{eqn:equivalent}
        \|\nabla v\|_2^2 + \|\Delta v\|_2^2 + \alpha \|(-\Delta)^{s} v\|_2^2 
    \sim \|v\|_{H^2(\mathbb{R}^N)}^2.
    \end{equation}
\end{lemma}
 
\begin{proof}
    To see this we use interpolation arguments \cite{CIRANT2019913,CANGIOTTI2024677}. 
    For all $\varepsilon > 0$, using 
    Plancherel's theorem in combination with H{\"o}lder's and Young's inequality we have 
    \begin{align} \label{eqn:interpolation-1}
        \|(-\Delta)^{s} v\|_2^2  
        &= \int_{\mathbb{R}^{N}} |\xi|^{4s} |\widehat{v}(\xi)|^2 d \xi 
        \leq \left(\int_{\mathbb{R}^{N}} |\widehat{v}(\xi)|^2 d \xi\right)^{1-s} 
        \left(\int_{\mathbb{R}^{N}} |\xi|^{4} |\widehat{v}(\xi)|^2 d \xi\right)^{s} \notag \\
        &= \|v\|_2^{2-2s} \|\Delta v\|_2^{2s} 
        \leq (1-s)\varepsilon^{-\frac{s}{1-s}} \|v\|_2^2 + s \varepsilon \|\Delta v\|_2^2.
    \end{align}
    Denote $\alpha^+= \max \{\alpha, 0 \}$ and $\alpha^- = \max \{-\alpha, 0 \}$. Then
    \begin{align} \label{eqn:equivalent-1} 
         \|\nabla v\|_2^2 +  \|\Delta v\|_2^2 + \alpha \|(-\Delta)^{s} v\|_2^2 
    & \leq  \alpha^+ (1-s)\varepsilon^{-\frac{s}{1-s}} \|v\|_2^2 + \|\nabla v\|_2^2 + 
    \alpha^+ s \varepsilon \|\Delta v\|_2^2 \notag \\
    & \lesssim  \|v\|_{H^2(\mathbb{R}^N)}^2.
    \end{align}
    When $\alpha > -\lambda_1^{1-s}(1-s)^{1-s}s^{-s}$, one can choose $\varepsilon > 0$ 
    such that 
    \[\lambda_1^{-\frac{1-s}{s}}(1-s)^{\frac{1-s}{s}}(\alpha^-)^{\frac{1-s}{s}} 
    < \varepsilon < \frac{1}{\alpha^- s},\]
    in which case $1-\alpha^- s \varepsilon > 0$ and 
    $\lambda_1 - \alpha^- (1-s) \varepsilon^{-\frac{1-s}{s}} > 0$ hold.  
    It follows from Poincar{\'e}'s inequality that 
    \begin{align} \label{eqn:equivalent-2} 
        \|\nabla v\|_2^2 + \|\Delta v\|_2^2 + \alpha \|(-\Delta)^{s} v\|_2^2 
    & \geq \left(\lambda_1 - \alpha^- (1-s) \varepsilon^{-\frac{1-s}{s}}\right) \|v\|_2^2 
    +  \left(1-\alpha^- s \varepsilon\right) \|\Delta v\|_2^2 \notag \\
    & \gtrsim \|v\|_{H^2(\mathbb{R}^N)}^2.
    \end{align}
    Combining \eqref{eqn:equivalent-1} with \eqref{eqn:equivalent-2}, we get 
    \eqref{eqn:equivalent}. 

    As for the case $\alpha > -(2-2s)^{2s-2}(2s-1)^{1-2s}$ along 
    with $s \in \left(\frac{1}{2},1\right)$, a similar process yields that 
    \[\|(-\Delta)^{s} v\|_2^2 \leq 
    (2-2s)\varepsilon^{-\frac{1-2s}{2-2s}} \|\nabla v\|_2^2 
    + (2s-1) \varepsilon \|\Delta v\|_2^2.\] 
    In this case we can choose $\varepsilon > 0$ such that 
    \[(2-2s)^{\frac{2-2s}{2s-1}}(\alpha^-)^{\frac{2-2s}{2s-1}} 
    < \varepsilon < \frac{1}{\alpha^- (2s-1)},\]
    then $1-\alpha^- (2s-1) \varepsilon > 0$ and 
    $1-\alpha^- (2-2s) \varepsilon^{\frac{1-2s}{2-2s}} > 0$ hold true. 
    Hence,  
    \begin{align} \label{eqn:equivalent-3} 
        \|\nabla v\|_2^2 + \|\Delta v\|_2^2 + \alpha \|(-\Delta)^{s} v\|_2^2 
    & \geq \left(1-\alpha^- (2-2s) \varepsilon^{\frac{1-2s}{2-2s}}\right) \|\nabla v\|_2^2 
    +  \left(1-\alpha^- (2s-1) \varepsilon\right) \|\Delta v\|_2^2 \notag \\
    & \gtrsim \|v\|_{H^2(\mathbb{R}^N)}^2,
    \end{align}
    then \eqref{eqn:equivalent-1} and \eqref{eqn:equivalent-3} conclude the proof.
\end{proof}

Setting 
\[\|v\|_{(\alpha)} \coloneqq \left(
    \|\nabla v\|_2^2 + \|\Delta v\|_2^2 + \alpha \|(-\Delta)^{s} v\|_2^2\right)^
    {\frac{1}{2}},\]
the above lemma implies that as long as (H2) is satisfied, 
$\|\cdot\|_{(\alpha)}$ and the 
usual $H^2(\mathbb{R}^N)$-norm are actually equivalent on the space 
$\mathcal{H}_0^2(\Omega)$. For $p$ satisfying (H1), 
we denote by $S_{p(\cdot)}$ a positive constant that 
the optimal embedding constant of $\mathcal{H}_0^2(\Omega) 
\hookrightarrow \mathcal{W}_0^{1,p}(\Omega)$, i.e. 
\[\frac{1}{S_{p(\cdot)}} = \inf_{v \in \mathcal{H}_0^2(\Omega) \backslash \{0\}} 
\frac{\|v\|_{(\alpha)}}{\|\nabla v\|_{p(\cdot)}}.\]

The Nehari manifold associated with $J(v;t)$ is given by 
\[\mathcal{N}(t) = \{v \in \mathcal{H}_0^2(\Omega) \backslash \{0\}: I(v;t) = 0\}.\]
The following lemma reveals that for any $t \geq 0$, $\mathcal{N}(t)$ is nonempty. 
\begin{lemma} \label{lem:nehari-nonempty} 
    Let (H1)-(H3) be fulfilled. For a given $v \in \mathcal{H}_0^2(\Omega)$ 
    with $\|v\|_{(\alpha)} \ne 0$ and each $t \geq 0$, there exists a constant
    $\mu^* = \mu^*(v, t) > 0$ such that $I(\mu^* v; t) = 0$.
\end{lemma}

\begin{proof}
    For any $\mu > 0$ and 
    fixed $t$, we set $h(\mu) \coloneqq I(\mu v; t)$. It is easy to see that $h(\mu)$ is 
    continuous with respect to $\mu$. Moreover, it holds that 
    \begin{align*}
        h(\mu) &\geq \mu^2 \|v\|_{(\alpha)}^2 - k(t) 
    \max \left\{\mu^{p^-} \|\nabla v\|_{p(\cdot)}^{p^-}, 
    \mu^{p^+} \|\nabla v\|_{p(\cdot)}^{p^+}\right\} \\
    &= \min \left\{ \mu^2 \|v\|_{(\alpha)}^2 - \mu^{p^-} k(t) \|\nabla v\|_{p(\cdot)}^{p^-}, 
    \mu^2 \|v\|_{(\alpha)}^2 - \mu^{p^+} k(t) \|\nabla v\|_{p(\cdot)}^{p^+} \right\}, 
    \end{align*}
    A simple calculation yields that there exists a positive constant 
    $\widehat{\mu}_1 = \widehat{\mu}_1(v,t)$ given by 
    \[\widehat{\mu}_1 \coloneqq \min \left\{\left(\frac{\|v\|_{(\alpha)}^2}
    {k(t) \|\nabla v\|_{p(\cdot)}^{p^-}}\right)^{\frac{1}{p^--2}}, 
    \left(\frac{\|v\|_{(\alpha)}^2}
    {k(t) \|\nabla v\|_{p(\cdot)}^{p^+}}\right)^{\frac{1}{p^+-2}}\right\}\]
    such that $h(\mu) \geq 0$ for $\mu \in (0, \widehat{\mu}_1]$. Similarly, it follows from 
    \[h(\mu) \leq \max \left\{ \mu^2 \|v\|_{(\alpha)}^2 - \mu^{p^-} k(t) \|\nabla v\|_{p(\cdot)}^{p^-}, 
    \mu^2 \|v\|_{(\alpha)}^2 - \mu^{p^+} k(t) \|\nabla v\|_{p(\cdot)}^{p^+} \right\}\]
    that there exists a positive constant 
    $\widehat{\mu}_2 = \widehat{\mu}_2(v,t) \geq \widehat{\mu}_1$ given by 
    \[\widehat{\mu}_2 \coloneqq \max \left\{\left(\frac{\|v\|_{(\alpha)}^2}
    {k(t) \|\nabla v\|_{p(\cdot)}^{p^-}}\right)^{\frac{1}{p^--2}}, 
    \left(\frac{\|v\|_{(\alpha)}^2}
    {k(t) \|\nabla v\|_{p(\cdot)}^{p^+}}\right)^{\frac{1}{p^+-2}}\right\}\]
    such that $h(\mu) \leq 0$ for $\mu \in [\widehat{\mu}_2,+\infty)$. Therefore, there exists a positive 
    constant $\mu^* = \mu^*(v,t)$ such that $\widehat{\mu}_1 \leq \mu^* \leq \widehat{\mu}_2$ and
    $I(\mu^* v; t) = 0$.
\end{proof} 

For each $t \geq 0$, $\mathcal{N}(t)$ separates the two unbounded sets
\[\begin{array}{c} \vspace{0.25em}
    \mathcal{N}_+(t) = \{v \in \mathcal{H}_0^2(\Omega) : I(v;t) > 0\} \cup \{0\}, \\
    \mathcal{N}_-(t) = \{v \in \mathcal{H}_0^2(\Omega) : I(v;t) < 0\}.
\end{array}\] 
We proceed to define the potential well for problem  
\eqref{eqn:main} by 
\[\begin{array}{c} \vspace{0.25em}
    W(t) = \{v \in \mathcal{H}_0^2(\Omega) : J(v;t) < d(t), I(v;t) > 0\} \cup \{0\}, \\
\end{array}\]
where 
\[d(t) = \inf_{v \in \mathcal{N}(t)} J(v ; t)\]
denotes the depth of the potential well $W(t)$. The nonemptiness of $\mathcal{N}(t)$ ensures that the potential well depth $d(t)$ is 
well-defined. We denote 
\[\underline{d} \coloneqq \inf_{t \geq 0} d(t), \quad 
\overline{d} \coloneqq \sup_{t \geq 0} d(t).\] 
The following lemma provides several important properties of $d(t)$.

\begin{lemma} \label{lem:depth-bound} 
    Let (H1)-(H3) be fulfilled. Then for any $t \geq 0$, the following statements hold true:
    \begin{enumerate}[(i)]
        \item It holds that
    \[d(t) \leq \frac{p^+-2}{2p^+} \max \left\{ 
        \left(k(t)S_{p(\cdot)}^{p^-}\right)^{\frac{2}{2-p^-}}, 
        \left(k(t)S_{p(\cdot)}^{p^+}\right)^{\frac{2}{2-p^+}}
    \right\},\] 
    and
    \[d(t) \geq \frac{p^--2}{2p^-} \min \left\{ 
        \left(k(t)S_{p(\cdot)}^{p^-}\right)^{\frac{2}{2-p^-}}, 
        \left(k(t)S_{p(\cdot)}^{p^+}\right)^{\frac{2}{2-p^+}}
    \right\}.\] 
        \item $0 \leq \underline{d} \leq d(t) \leq \overline{d} < + \infty$. Additionally, 
        if $k(t)$ is bounded, then $\underline{d} > 0$.
        \item For any $v \in \mathcal{N}_-(t)$, 
        \begin{align} \label{eqn:unstable-depth-upper}
            d(t) \leq \min \left\{\frac{p^+ - 2}{2 p^+} \|v\|_{(\alpha)}^2, 
            k(t) \int_{\Omega} \frac{p(x)-2}{2p(x)} |\nabla v|^{p(x)} dx \right\}.
        \end{align}
    \end{enumerate}
\end{lemma}

\begin{proof}
    We first prove (i). Fix $v \in \mathcal{H}_0^2(\Omega) \backslash \{0\}$. 
    For each $t \geq 0$, Lemma \ref{lem:nehari-nonempty} implies that
    there exists a positive constant $\mu^* = \mu^*(v,t)$ such that 
    $\mu^* v \in \mathcal{N}(t)$. 
    Therefore, we have 
    \begin{align} \label{eqn:depth-upper-bound-critical} 
        d(t) \leq J( v;t) &\leq \frac{1}{2} \|\mu^*v\|_{(\alpha)}^2 
        -  \frac{k(t)}{p^+} \int_{\Omega} |\nabla (\mu^*v)|^{p(x)} dx \notag \\
        & = \frac{1}{p^+} I(\mu^* v;t) + 
        \frac{p^+-2}{2 p^+} (\mu^*)^2 \| v\|_{(\alpha)}^2 \notag \\ 
        & = \frac{p^+-2}{2 p^+} (\mu^*)^2 \| v\|_{(\alpha)}^2.
    \end{align}
    It follows from the proof of Lemma \ref{lem:depth-bound} that 
    \begin{align*}
        d(t) & \leq \frac{p^+-2}{2 p^+} \|v\|_{(\alpha)}^2  \widehat{\mu}_2^2 \\
        & = \frac{p^+-2}{2 p^+} \max \left\{\left(\frac{\|v\|_{(\alpha)}^{p^-}}
        {k(t) \|\nabla v\|_{p(\cdot)}^{p^-}}\right)^{\frac{2}{p^--2}}, 
        \left(\frac{\|v\|_{(\alpha)}^{p^+}}
        {k(t) \|\nabla v\|_{p(\cdot)}^{p^+}}\right)^{\frac{2}{p^+-2}}\right\}.
    \end{align*}
    By the arbitrariness of $v \in \mathcal{H}_0^2(\Omega) \backslash \{0\}$ and 
    the definition of $S_{p(\cdot)}$, it follows that for any $t \geq 0$, 
    \[d(t) \leq \frac{p^+-2}{2p^+} \max \left\{ 
        \left(k(t)S_{p(\cdot)}^{p^-}\right)^{\frac{2}{2-p^-}}, 
        \left(k(t)S_{p(\cdot)}^{p^+}\right)^{\frac{2}{2-p^+}}
    \right\}.\]

    Now we choose $v \in \mathcal{N}(t)$. Then 
    \begin{align} \label{eqn:norm-lower-bound} 
        \|v\|_{(\alpha)}^2 &= k(t) \int_{\Omega} |\nabla v|^{p(x)} dx
        \leq k(t) \max \left\{\|\nabla  v\|_{p(\cdot)}^{p^-}, 
        \|\nabla  v\|_{p(\cdot)}^{p^+}\right\} \notag \\
        & \leq k(t) \max \left\{S_{p(\cdot)}^{p^-} \|v\|_{(\alpha)}^{p^-}, 
        S_{p(\cdot)}^{p^+} \|v\|_{(\alpha)}^{p^+} \right\},
    \end{align}
    from which it follows that 
    \begin{align} \label{eqn:norm-lower-bound-2} 
    \|v\|_{(\alpha)}^2 \geq \min \left\{ 
        \left(k(t)S_{p(\cdot)}^{p^-}\right)^{\frac{2}{2-p^-}}, 
        \left(k(t)S_{p(\cdot)}^{p^+}\right)^{\frac{2}{2-p^+}}
    \right\}.
    \end{align}
    Since $v \in \mathcal{N}(t)$, we get 
    \begin{align*}
        J(v;t) &= \frac{1}{2} \|v\|_{(\alpha)}^2 
        -  k(t) \int_{\Omega} \frac{1}{p(x)} |\nabla v|^{p(x)} dx \\
        & \geq \frac{1}{2} \|v\|_{(\alpha)}^2 - 
        \frac{k(t)}{p^-} \int_{\Omega} |\nabla v|^{p(x)} dx 
        = \frac{p^- - 2}{2 p^-} \|v\|_{(\alpha)}^2 \\
        & \geq \frac{p^--2}{2p^-} \min \left\{ 
            \left(k(t)S_{p(\cdot)}^{p^-}\right)^{\frac{2}{2-p^-}}, 
            \left(k(t)S_{p(\cdot)}^{p^+}\right)^{\frac{2}{2-p^+}}
        \right\}.
    \end{align*}
    The definition of $d(t)$ concludes the proof of (i).

    As for (ii), since $k(t)$ is nondecreasing, we obtain that
    \[0<d(t) \leq \frac{p^+-2}{2p^+} \max \left\{ 
        \left(k(0)S_{p(\cdot)}^{p^-}\right)^{\frac{2}{2-p^-}}, 
        \left(k(0)S_{p(\cdot)}^{p^+}\right)^{\frac{2}{2-p^+}}
    \right\},\]
    hence, $0 \leq \underline{d} \leq \overline{d} < + \infty$. If $k(t)$ is bounded, then 
    \[\underline{d} \geq \frac{p^--2}{2p^-} \min \left\{ 
        \left(k(+\infty)S_{p(\cdot)}^{p^-}\right)^{\frac{2}{2-p^-}}, 
        \left(k(+\infty)S_{p(\cdot)}^{p^+}\right)^{\frac{2}{2-p^+}}
    \right\} > 0.\]

    Finally, let us prove (iii). Recall the function $h(\mu)$ defined in the proof of 
    Lemma \ref{lem:nehari-nonempty}. For any given $v \in \mathcal{N}_-(t)$, 
    we know that $h(1) < 0$. A direct computation shows that 
    \[h(\mu) \geq \mu^2 \|v\|_{(\alpha)}^2 - \mu^{p^-} k(t) \int_{\Omega} |\nabla v|^{p(x)} dx\]
    holds for $\mu \in (0,1)$, which implies that there exists a positive constant 
    $\mu_0 = \mu_0(v,t) < 1$ such that $h(\mu_0) > 0$. Therefore,
    there exists is a positive constant $\mu^* = \mu^*(v,t) \in (0,1)$ 
    such that $\mu^* v \in \mathcal{N}(t)$. It follows from  
    \eqref{eqn:depth-upper-bound-critical} that 
    \begin{equation*}
        d(t) \leq J( v;t) \leq 
        \frac{p^+-2}{2 p^+} (\mu^*)^2 \| v\|_{(\alpha)}^2
        \leq \frac{p^+-2}{2 p^+} \|v\|_{(\alpha)}^2.
    \end{equation*}
    On the other hand,
    \begin{align*}
        d(t) \leq J(\mu^* v;t) &= \frac{1}{2} I(\mu^* v;t) 
        + k(t) \int_{\Omega} \frac{p(x)-2}{2p(x)} 
        (\mu^*)^{p(x)} |\nabla v|^{p(x)} dx \\ 
        & \leq k(t) \int_{\Omega} \frac{p(x)-2}{2p(x)} |\nabla v|^{p(x)} dx,
    \end{align*}
    which concludes the proof of (iii).
\end{proof}    

For any $\varsigma > d(t)$, we define 
the sublevels of $J(\cdot ; t)$:
\[J^{\varsigma}(t) = \{v \in \mathcal{H}_0^2(\Omega) : J(v ; t) \leq \varsigma\}.\]
From the definition of $J(\cdot ; t)$, $\mathcal{N}(t)$, $d(t)$ and $J^{\varsigma}(t)$,
we know that for any $\varsigma > d(t)$, 
\[\mathcal{N}_{\varsigma}(t) \coloneqq \mathcal{N}(t) \cap J^{\varsigma}(t) = 
\{v \in \mathcal{N}(t): J(v ; t) \leq \varsigma\} \neq \varnothing.\]
We also define 
\begin{equation*}
    \lambdabar_{\varsigma}(t) \coloneqq \inf_{v \in \mathcal{N}_{\varsigma}(t)} \|v\|_2^2, \quad
    \underline{\lambdabar}_{\varsigma} \coloneqq \inf_{t \geq 0} \lambdabar_{\varsigma}(t), \quad 
    \lambdabar_{\infty}(t) \coloneqq \inf_{v \in \mathcal{N}(t)} \|v\|_2^2, \quad
    \underline{\lambdabar}_{\infty} \coloneqq \inf_{t \geq 0} \lambdabar_{\infty}(t).
\end{equation*}
It is clear that $\underline{\lambdabar}_{\varsigma}<+\infty$ if $\varsigma > \overline{d}$. 
Moreover, it holds that $0 \leq \lambdabar_{\infty}(t) \leq \lambdabar_{\varsigma}(t)$ for 
any $t \geq 0$, and 
$0 \leq \underline{\lambdabar}_{\infty} \leq \underline{\lambdabar}_{\varsigma}$.

Next, we discuss some properties of $\mathcal{N}_+(t)$, $\mathcal{N}_-(t)$, 
$\mathcal{N}_{\varsigma}(t)$ and $\lambdabar_{\varsigma}(t)$.

\begin{lemma} \label{lem:0-interior} 
    Let (H1)-(H3) be fulfilled. Then for any $t \geq 0$, $0$ is an interior of 
    $\mathcal{N}_+(t)$. Moreover, for any given $T' \in (0,+\infty)$, there exists 
    $\varepsilon>0$ depending only on $T'$ such that for any $t \in [0,T']$ and any 
    $v \in \mathcal{H}_0^2(\Omega)$ satisfying $\|v\|_{(\alpha)} < \varepsilon$, 
    it holds that $v \in \mathcal{N}_+(t)$.
\end{lemma}

\begin{proof}
    Choose $v \in \mathcal{H}_0^2(\Omega)$ with $0<\|v\|_{(\alpha)}<1$. Set 
    $v_1 = \frac{v}{\|v\|_{(\alpha)}}$ and 
    $\widetilde{S}_{p(\cdot)} = \max \left\{S_{p(\cdot)}^{p^-}, S_{p(\cdot)}^{p^+} \right\}$.
    For $\mu \in (0,1)$ and $t \geq 0$, it follows that 
    \[I(\mu v_1; t)  \geq \mu^2 \|v_1\|_{(\alpha)}^2 - \mu^{p^-} k(t) 
        \int_{\Omega} |\nabla v_1|^{p(x)} dx 
        \geq \mu^2 \left(1 - \mu^{p^--2} k(t) \widetilde{S}_{p(\cdot)}\right).\]
    Taking $\mu = \|v\|_{(\alpha)}$, there exists 
    $\varepsilon(t) = \min \left\{1,\left(k(t)\widetilde{S}_{p(\cdot)}\right)^{\frac{1}{2-p^-}}\right\}$ 
    such that for $v \in \mathcal{H}_0^2(\Omega)$ with $0<\|v\|_{(\alpha)}<\varepsilon(t)$,
    $I(v;t)>0$, which implies that for any $t \geq 0$, $0$ is an interior of 
    $\mathcal{N}_+(t)$. To show the second statement one only needs to take 
    $\varepsilon=\varepsilon(T')$.
\end{proof}  

\begin{lemma} \label{lem:away-from-0} 
    Let (H1)-(H3) be fulfilled. Then for each $t \geq 0$, the following statements are true:
    \begin{enumerate}[(i)]
        \item $\operatorname{dist}(0,\mathcal{N}(t)) > 0$, 
        $\operatorname{dist}(0,\mathcal{N}_-(t)) > 0$.
        \item For $\varsigma > d(t)$, the sets $\mathcal{N}_+(t) \cap J^{\varsigma}(t)$ and 
        $\mathcal{N}_{\varsigma}(t)$ are both bounded in $\mathcal{H}_0^2(\Omega)$.
    \end{enumerate}   
\end{lemma}

\begin{proof}
    Firstly, we prove (i). For $v \in \mathcal{N}(t)$, 
    from \eqref{eqn:norm-lower-bound-2}, we have
    \[\operatorname{dist}(0,\mathcal{N}(t)) 
    \sim \inf_{v \in \mathcal{N}(t)} \|v\|_{(\alpha)} 
    \geq \min \left\{ 
        \left(k(t)S_{p(\cdot)}^{p^-}\right)^{\frac{1}{2-p^-}}, 
        \left(k(t)S_{p(\cdot)}^{p^+}\right)^{\frac{1}{2-p^+}}
    \right\} > 0.\] 
    For $v \in \mathcal{N}_-(t)$, since
    \[\|v\|_{(\alpha)}^2 - k(t) \int_{\Omega} |\nabla v|^{p(x)} dx = I(v;t)<0,\]
    a similar process as in \eqref{eqn:norm-lower-bound} leads 
    to \eqref{eqn:norm-lower-bound-2}, and thus 
    \[\operatorname{dist}(0,\mathcal{N}_-(t)) 
    \sim \inf_{v \in \mathcal{N}_-(t)} \|v\|_{(\alpha)} 
    \geq \min \left\{ 
        \left(k(t)S_{p(\cdot)}^{p^-}\right)^{\frac{1}{2-p^-}}, 
        \left(k(t)S_{p(\cdot)}^{p^+}\right)^{\frac{1}{2-p^+}}
    \right\} > 0.\]

    To see (ii), we choose $v \in J^{\varsigma}(t)$ so that $I(v;t) \geq 0$. Then 
    we have 
    \begin{align*}
        \varsigma \geq J(v;t) & \geq \frac{1}{2} \|v\|_{(\alpha)}^2 - 
        \frac{k(t)}{p^-} \int_{\Omega} |\nabla v|^{p(x)} dx \\
        &= \frac{p^- - 2}{2 p^-} \|v\|_{(\alpha)}^2 + \frac{1}{p^-} I(v;t) \\
        &\geq \frac{p^- - 2}{2 p^-} \|v\|_{(\alpha)}^2.
    \end{align*}
    Hence, 
    \[\|v\|_{(\alpha)}^2 \leq \frac{2 p^- \varsigma}{p^- - 2},\]
    which concludes the proof of (ii).
\end{proof}

\begin{lemma} \label{lem:2-norm-inf-lower} 
    Let (H1)-(H3) be fulfilled. Then:
    \begin{enumerate}[(i)]
        \item For each $t \geq 0$ and any $\varsigma > d(t)$, 
        $\lambdabar_{\varsigma}(t) > 0$. Additionally, if $k(t)$ is bounded 
        and $\varsigma > \overline{d}$, then $\underline{\lambdabar}_{\varsigma}>0$.
        \item If $p^+ \leq \frac{2(N+4)}{N+2}$, then $\lambdabar_{\infty}(t) > 0$ for any
        $t \geq 0$. Furthermore, if $k(t)$ is bounded, 
        then $\underline{\lambdabar}_{\infty}>0$.
    \end{enumerate}
\end{lemma}

\begin{proof}
    Fix $t \geq 0$ and choose $v \in \mathcal{N}(t)$. 
    It follows from the Gagliardo-Nirenberg inequality (see \cite{ADAMS200379}), 
    the embedding $L^{p^+}(\Omega) \hookrightarrow L^{p(\cdot)}(\Omega)$ and 
    $\|\cdot\|_{(\alpha)} \sim \|\cdot\|_{H^2(\mathbb{R}^N)}$
    that there exists a positive constant $\widetilde{C}_1$ such that 
    \begin{equation} \label{eqn:gagliardo-1}
        \|\nabla v\|_{p(\cdot)} \leq (|\Omega|+1) \|\nabla v\|_{p^+} 
        \leq \widetilde{C}_1 \|v\|_{(\alpha)}^{\theta} \|v\|_2^{1-\theta}, 
    \end{equation}
    where $\theta = \frac{1}{2} + \frac{N}{4}\left(1-\frac{2}{p^+}\right)$. 
    Due to the assumption (H1), it is easy to see 
    that $\theta \in \left(\frac{1}{2},1\right)$. Since $v \in \mathcal{N}(t)$, we have 
    \begin{align} \label{eqn:px-norm-lower}
        \|\nabla v\|_{p(\cdot)} & \geq \min \left\{ 
            \left(\int_{\Omega} |\nabla v|^{p(x)} dx\right)^{\frac{1}{p^-}}, 
            \left(\int_{\Omega} |\nabla v|^{p(x)} dx\right)^{\frac{1}{p^+}} 
        \right\} \notag \\
        & = \min \left\{ 
            \left(k(t)\right)^{-\frac{1}{p^-}} \|v\|_{(\alpha)}^{\frac{2}{p^-}}, 
            \left(k(t)\right)^{-\frac{1}{p^+}} \|v\|_{(\alpha)}^{\frac{2}{p^+}}
        \right\}.
    \end{align}
    Combining \eqref{eqn:gagliardo-1} and \eqref{eqn:px-norm-lower}, we obtain that
    \begin{equation} \label{eqn:2-norm-lower}
        \|v\|_2 \geq \min \left\{ 
        \left[\frac{1}{\widetilde{C}_1} 
        \left(k(t)\right)^{-\frac{1}{p^-}} \|v\|_{(\alpha)}^{\frac{2}{p^-}-\theta}
        \right]^{\frac{1}{1-\theta}}, 
        \left[\frac{1}{\widetilde{C}_1} 
        \left(k(t)\right)^{-\frac{1}{p^+}} \|v\|_{(\alpha)}^{\frac{2}{p^+}-\theta}
        \right]^{\frac{1}{1-\theta}}
    \right\}.
    \end{equation} 
    
    (i) Let $v \in \mathcal{N}_{\varsigma}(t)$. From Lemma \ref{lem:away-from-0}, it follows 
    that the right-hand side of \eqref{eqn:2-norm-lower} remains bounded 
    away from $0$ for each fixed $t \geq 0$ no 
    matter what the signs of $\frac{2}{p^-}-\theta$ and $\frac{2}{p^+}-\theta$ are. 
    Thus, $\lambdabar_{\varsigma}(t) > 0$ for any $t \geq 0$. Moreover,
    if $k(t)$ is bounded and $\varsigma > \overline{d}$, the right side of 
    \eqref{eqn:2-norm-lower} will have a uniform positive lower bound, which leads 
    to $\underline{\lambdabar}_{\varsigma}>0$.

    (ii) When $p^+ \leq \frac{2(N+4)}{N+2}$, it holds 
    that $\frac{2}{p^-}-\theta \geq \frac{2}{p^+}-\theta \geq 0$. 
    In this case, Lemma \ref{lem:away-from-0}(i) alone is still sufficient to conclude 
    that the right-hand side of \eqref{eqn:2-norm-lower} remains bounded away from $0$. 
    Therefore, $\lambdabar_{\infty}(t) > 0$ for any $t \geq 0$. 
    When $k(t)$ is bounded, the right-hand side of \eqref{eqn:2-norm-lower} still has 
    a uniform positive lower bound, hence $\underline{\lambdabar}_{\infty}>0$.
\end{proof}

In the study of blow-up properties, we consider solutions that satisfy a conservation law, 
as defined in the following.

\begin{definition}[Strong solution] \label{def:strong}
    A function 
    $u \in C([0,T];\mathcal{W}_0^{1,p(\cdot)}(\Omega))$ is said to be a 
    (strong) solution to problem \eqref{eqn:main} on $[0,T]$ if it is a 
    weak solution to problem \eqref{eqn:main} and the following conservation law 
    \begin{equation} \label{eqn:conservation-law}
        J(u(t);t) + \int_{0}^{t} \|\partial_t u(\tau)\|_2^2 d \tau 
        + \int_{0}^{t} \int_{\Omega} \frac{k'(\tau)}{p(x)}|\nabla u(\tau)|^{p(x)} dx d \tau 
        = J(u_0;0) 
    \end{equation}
    holds for $t \in [0,T]$.
\end{definition}

\begin{remark}
    (i) It is clear that any solution $u(t)$ to problem \eqref{eqn:main} on $[0,T]$ satisfies 
    $J(u(t);t) \in C([0,T])$. Combined with Lemma \ref{lem:modular-continuity}, this yields 
    $\|u(t)\|_{(\alpha)}^2 \in C([0,T])$ and $I(u(t);t) \in C([0,T])$. 

    (ii) One may obtain \eqref{eqn:conservation-law} by taking $\varphi = u_t$ in 
    \eqref{eqn:main-weak}. Unfortunately, this argument is only formal as it requires 
    $u$ to be sufficiently regular so that $u_t \in L^2(0,T;\mathcal{H}_0^2(\Omega))$. 
    Standard approximation arguments yield only an energy inequality, 
    namely \eqref{eqn:energy-inequality} presented in the next section. The conservation 
    law \eqref{eqn:conservation-law} is so essential that the analysis of blow-up properties, 
    both in this paper and in many previous works 
    \cite{LI201696,DONG201789,HAN2018451,DING20201046,GUO202245,ZHAO202412}, 
    is actually carried out for stronger solutions like those defined in 
    Definition \ref{def:strong} rather than for weak solutions directly. We emphasize this 
    distinction here in the interest of mathematical rigor; however, this does not diminish 
    the value of our work or those of prior works on blow-up behavior.
\end{remark}

To close this section, we provide the definitions of the maximal existence time and 
finite time blow-up for solutions to problem \eqref{eqn:main}.

\begin{definition}[Maximal existence time]
    Suppose that $u(t)$ is a weak (or strong) solution to problem \eqref{eqn:main}. We 
    say that a positive $T_*$ is the maximal existence time of $u(t)$ as follows:
    \begin{enumerate}[(i)]
        \item If for all $T' \in (0,+\infty)$, $u(t)$ exists on $[0,T']$, then 
        $T_* = +\infty$. In this case, $u$ is said to be global.
        \item If there exists a $T' \in (0,+\infty) $ such that $u(t)$ exists 
        on $[0,T]$ for all $T \in (0,T')$, but does not exist on $[0,T']$, then $T_*=T'$.
    \end{enumerate}
\end{definition}

\begin{definition}[Finite time blow-up]
    Suppose that $u(t)$ is a weak (or strong) solution to problem \eqref{eqn:main}. We say 
    that $u(t)$ blows up in finite time if its maximal existence time $T_*$ 
    is finite and 
    \[\lim_{t \nearrow  T_*} \|u(t)\|_{H^2(\mathbb{R}^N)} = + \infty.\]
\end{definition}

\section{Local well-posedness} \label{sec3} 
In this section, we will focus on the local well-posedness of 
problem \eqref{eqn:main} on $[0,T']$, where $T' \in (0,+\infty)$ can be chosen 
arbitrarily large. We henceforth assume that (H1)-(H3) are satisfied.

\begin{theorem} \label{thm:local-well-posedness}
    Given $u_0 \in \mathcal{H}_0^2(\Omega)$, there exists $T>0$ depending only on 
    $\|\Delta u_0\|_2$ such that 
    problem \eqref{eqn:main} possesses a unique weak solution $u$ on $[0,T]$ 
    with $u \in C([0,T];\mathcal{W}_0^{1,p(\cdot)}(\Omega)) \cap 
    C_w([0,T];\mathcal{H}_0^2(\Omega))$. Moreover, the following energy inequality 
    holds for a.e. $t \in [0,T]$:
    \begin{align} \label{eqn:energy-inequality}
        J(u(t);t) + \int_{0}^{t} \|\partial_t u(\tau)\|_2^2 d \tau 
        + \int_{0}^{t} \int_{\Omega} \frac{k'(\tau)}{p(x)}|\nabla u(\tau)|^{p(x)} dx d \tau 
        \leq J(u_0;0). 
    \end{align}
\end{theorem}

\begin{proof}
    The proof relies on the standard Galerkin method \cite{LADYENSKAJA1968}, 
    as seen in works such as \cite{LI201696,MIYAKE2021247}, and we outline the process here. 

    \textbf{Step 1. Construct approximating solutions.} 
    We consider the orthonormal basis $\{\omega_i\}_{i=1}^{\infty}$ of 
    $\mathcal{H}_0(\Omega)$ formed by the eigenfunctions of $-\Delta$ with 
    homogeneous Dirichlet boundary conditions, alongside their nondecreasing sequence of 
    positive eigenvalues $\{\lambda_i\}_{i=1}^{\infty}$:
    \[\left\{
        \begin{array}{ll}
            -\Delta \omega_i = \lambda_i \omega_i & \textrm{in}~~ \Omega,  \\ 
            \omega_i = 0, & \textrm{in}~~ \mathbb{R}^N \backslash \Omega , 
        \end{array} 
    \right. \quad \|\omega_i\|_2 = 1, \quad 0 < \lambda_i \leq \lambda_{i+1}, 
    \quad i = 1,2, \dots.\]
    For each $m \in \mathbb{N}$ and $i = 1,2, \cdots, m$, the solution 
    $g^m(t) = \left(g_1^m(t), g_2^m(t), \cdots, g_m^m(t)\right)$ to the following Cauchy 
    problem for an ordinary differential system with respect to $t$ is considered:
    \begin{equation} \label{eqn:ode-system}
        \left\{
        \begin{array}{l} 
            \displaystyle \frac{d}{dt} g_i^m(t) + L_i\left(g^m(t)\right)
           = k(t) \int_{\Omega} \Big|\sum_{j=1}^m g_j^m(t) \nabla \omega_j \Big|^{p(x)-2} 
            \sum_{j=1}^m g_j^m(t) \nabla \omega_j \cdot \nabla \omega_i dx, \\ 
            g_i^m(0) = c_i^m \coloneqq (u_0, \omega_i),
        \end{array} 
        \right.
    \end{equation}
    where $L_i$ are linear functions of $g_1^m(t), g_2^m(t), \cdots, g_m^m(t)$ given by 
    \[L_i\left(g^m(t)\right)=\lambda_i^2 g_i^m(t) + \sum_{j=1}^m
    \left[\alpha \left((-\Delta)^s \omega_j, (-\Delta)^s \omega_i\right)
    + \int_{\Omega} \nabla \omega_j \cdot \nabla \omega_i d x\right] g_j^m(t).\]
    By Peano's theorem (see \cite{MATTER1982}) we know that there exist 
    a maximal existence time $T_m \in (0,T']$ 
    and a unique local solution $g^m(t) \in C^1\left([0,T_m); \mathbb{R}^m\right)$ to 
    \eqref{eqn:ode-system}. The local solutions can be extended uniquely to global ones 
    if $g_j^m(t)$ is bounded for  $j = 1,2,\dots,m$. The
    approximate solutions to the problem \eqref{eqn:main} are given by 
    \[u^m(t,x) = \sum_{j=1}^{m} g_j^m(t) \omega_j(x), \quad m = 1,2, \cdots, \quad
    (t,x) \in [0,T_m) \times \mathbb{R}^N,\]
    which satisfy 
    \[u^m(0,x) = \sum_{j=1}^{m} c_j^m \omega_j(x) \to u_0(x)\]
    strongly in $\mathcal{H}_0^2(\Omega)$ as $m \to \infty$.

    \textbf{Step 2. A priori estimates.} 
    Fix $m \in \mathbb{N}$. The approximate solution $u^m$ satisfies 
    \begin{align} \label{eqn:integral-identity-base}
    &\frac{d}{dt} (u^m(t), \omega_i) + (\Delta u^m(t), \Delta \omega_i) + 
    \alpha ((-\Delta)^s u^m(t), (-\Delta)^s \omega_i) 
    + \int_{\Omega} \nabla u^m(t) \cdot \nabla \omega_i d x \notag \\ 
    &= k(t) \int_{\Omega} |\nabla u^m(t)|^{p(x)-2} \cdot \nabla \omega_i d x    
    \end{align}
    for $t \in [0,T_m)$ and $i = 1,2, \cdots, m$. Multiplying 
    \eqref{eqn:integral-identity-base} by $\partial_t g_i^m(t)$, 
    summing over $i = 1,2,\cdots,m$ and integrating it on $(0,t)$, we obtain that
    \begin{equation} \label{eqn:decreasing-integral-um}
        J(u^m(t);t) + \int_{0}^{t} \|\partial_t u^m(\tau)\|_2^2 d \tau 
        + \int_{0}^{t} \int_{\Omega} \frac{k'(\tau)}{p(x)}|\nabla u^m(\tau)|^{p(x)} dx d \tau 
        = J(u^m(0);0),
    \end{equation}
    which implies that $J(u^m(t);t)$ is nonincreasing with respect to $t \in [0,T_m)$. 
    It follows from the convergence $u^m(0,x) \to u_0(x)$ in $\mathcal{H}_0^2(\Omega)$ 
    that $J(u^m(0);0) \to J(u_0;0)$ as $m \to \infty$. Set 
    \[K_1 \coloneqq \sup_{m \in \mathbb{N}} J(u^m(0);0),\]
    then $K_1$ is finite and independent of $m$. From the definition of $u^m(0)$, we have
    \[\|\Delta u^m(0)\|_2^2 = \sum_{j=1}^{m} (\lambda_j u_0, \omega_j)^2 = 
      \sum_{j=1}^{m} (\Delta u_0,  \Delta(\lambda_j^{-1}\omega_j))^2 
      \leq \|\Delta u_0\|_2^2.\]
    Due to $\|\cdot\|_{(\alpha)} \sim \|\cdot\|_{H^2(\mathbb{R}^N)}$, there exists 
    a constant $B_1>0$ such that 
    \[\|u^m(0)\|_{(\alpha)} \leq B_1 \|\Delta u^m(0)\|_2 \leq B_1 \|\Delta u_0\|_2.\]
    Set $B_{p(\cdot)} \coloneqq B_1 S_{p(\cdot)}$. Choose 
    $K_2 > \max \left\{B_{p(\cdot)}^{p^-}\|\Delta u_0\|_2^{p^-}, 
    B_{p(\cdot)}^{p^+}\|\Delta u_0\|_2^{p^+}\right\}$, then 
    \begin{align*}
        \int_{\Omega} |\nabla u^m(0)|^{p(x)} dx 
        &\leq \max \left\{\|\nabla u^m(0)\|_{p(\cdot)}^{p^-}, 
        \|\nabla u^m(0)\|_{p(\cdot)}^{p^+} \right\} \\
        &\leq \max \left\{S_{p(\cdot)}^{p^-} \|u^m(0)\|_{(\alpha)}^{p^-}, 
        S_{p(\cdot)}^{p^+} \|u^m(0)\|_{(\alpha)}^{p^+} \right\} \\
        &\leq \max \left\{B_{p(\cdot)}^{p^-} \|\Delta u_0\|_2^{p^-}, 
        B_{p(\cdot)}^{p^+} \|\Delta u_0\|_2^{p^+}\right\} < K_2.
    \end{align*}
    Set 
    \[\widetilde{T}_m \coloneqq \sup \left\{ \tau \in (0,T_m]: \sup_{t \in [0,\tau]}
    \int_{\Omega} |\nabla u^m(t)|^{p(x)} dx \leq K_2\right\},\]
    then $\widetilde{T}_m>0$, since 
    $u^m \in C\big([0,T_m);\mathcal{W}_0^{1,p(\cdot)}(\Omega)\big)$. For 
    $t \in [0,\widetilde{T}_m]$, it holds that
    \begin{equation} \label{eqn:a-priori-tm-1}
        \|u^m(t)\|_{(\alpha)}^2 = 2J(u^m(t);t)+2k(t)
    \int_{\Omega} \frac{1}{p(x)}|\nabla u^m(t)|^{p(x)} dx 
    \leq 2K_1 + \frac{2k(t)}{p^-}K_2.
    \end{equation}
    and
    \begin{align} \label{eqn:a-priori-tm-2}
       \int_{0}^{t} \|\partial_t u^m(\tau)\|_2^2 d \tau &= J(u^m(0);0)-J(u^m(t);t) 
        - \int_{0}^{t} \int_{\Omega} \frac{k'(\tau)}{p(x)}|\nabla u^m(\tau)|^{p(x)} dx d \tau \notag \\ 
        & \leq J(u^m(0);0) - \frac{1}{2}\|u^m(t)\|_{(\alpha)}^2 
        + k(t) \int_{\Omega} \frac{1}{p(x)}|\nabla u^m(t)|^{p(x)} dx \notag \\
        & \leq K_1 + \frac{k(t)}{p^-}K_2.
    \end{align}
    By using $\|\cdot\|_{(\alpha)} \sim \|\cdot\|_{H^2(\mathbb{R}^N)}$ again 
    and combining with \eqref{eqn:a-priori-tm-1}, recalling (H3), it follows that 
    there exists a constant $c_1>0$ such that 
    \begin{equation} \label{eqn:diff-delta-norm}
        \|\Delta u^m(t)-\Delta u^m(0)\|_2 \leq c_1 \|u^m(t)-u^m(0)\|_{(\alpha)} 
        \leq 2c_1 \left(2K_1 + \frac{2k(\widetilde{T}_m)}{p^-}K_2\right)^{\frac{1}{2}}
    \end{equation} 
    for $t \in [0,\widetilde{T}_m]$. From \eqref{eqn:a-priori-tm-2} it follows that 
    \begin{equation} \label{eqn:diff-norm}
        \|u^m(t)-u^m(0)\|_2^2 
        = \int_{\Omega} \left(\int_{0}^{t} \partial_t u^m(\tau) d \tau\right)^2 dx 
        \leq t \int_{0}^{t} \|\partial_t u^m(\tau)\|_2^2 d \tau 
        \leq \left(K_1 + \frac{k(\widetilde{T}_m)}{p^-}K_2\right) \widetilde{T}_m
    \end{equation} 
    for $t \in [0,\widetilde{T}_m]$. The Gagliardo-Nirenberg inequality and 
    the embedding $L^{p^+}(\Omega) \hookrightarrow L^{p(\cdot)}(\Omega)$ indicate 
    that there exists a constant $\widetilde{C}_2>0$ such that 
    \begin{align} \label{eqn:gagliardo-2}
        \|\nabla u^m(t) - \nabla u^m(0)\|_{p(\cdot)} 
        & \leq (|\Omega|+1) \|\nabla u^m(t) - \nabla u^m(0)\|_{p^+} \notag \\
        & \leq \widetilde{C}_2 \|\Delta u^m(t) - \Delta u^m(0)\|_2^{\theta} 
        \|u^m(t) - u^m(0)\|_2^{1-\theta}, 
    \end{align}
    for $t \in [0,\widetilde{T}_m]$, where 
    $\theta = \frac{1}{2} + \frac{N}{4}\left(1-\frac{2}{p^+}\right)$. Combining with
    \eqref{eqn:diff-delta-norm}, \eqref{eqn:diff-norm} and \eqref{eqn:gagliardo-2}, 
    we deduce that 
    \begin{align} 
        & \min \left\{ \left(\int_{\Omega} |\nabla u^m(t)|^{p(x)} dx\right)^{\frac{1}{p^-}}, 
        \left(\int_{\Omega} |\nabla u^m(t)|^{p(x)} dx\right)^{\frac{1}{p^+}}\right\} \notag \\
        & \leq \|\nabla u^m(t)\|_{p(\cdot)} \leq \|\nabla u^m(0)\|_{p(\cdot)} 
        + \|\nabla u^m(t) - \nabla u^m(0)\|_{p(\cdot)} \notag \\
        & \leq B_{p(\cdot)} \|\Delta u_0\|_2 + 8^{\frac{\theta}{2}}c_1^{\theta} \widetilde{C}_2 
        \left(K_1 + \frac{k(\widetilde{T}_m)}{p^-}K_2\right)^{\frac{1}{2}}
        \widetilde{T}_m^{\frac{1-\theta}{2}}
    \end{align}
    for $t \in [0,\widetilde{T}_m]$. 
    Setting $\widetilde{K}_2 \coloneqq \min \left\{ K_2^{\frac{1}{p^-}}, K_2^{\frac{1}{p^+}}\right\}$, 
    the definition of $\widetilde{T}_m$ implies that 
    \[0 < \widetilde{K}_2 - B_{p(\cdot)} \|\Delta u_0\|_2 \leq 8^{\frac{\theta}{2}}c_1^{\theta} \widetilde{C}_2 
    \left(K_1 + \frac{k(\widetilde{T}_m)}{p^-}K_2\right)^{\frac{1}{2}}
    \widetilde{T}_m^{\frac{1-\theta}{2}}.\]
    (H3) ensures that there exists $T \in (0,\widetilde{T}_m]$ independent of $m$ such that
    \[8^{\frac{\theta}{2}}c_1^{\theta} \widetilde{C}_2 
    \left(K_1 + \frac{k(T)}{p^-}K_2\right)^{\frac{1}{2}} T^{\frac{1-\theta}{2}} = 
    \widetilde{K}_2 - B_{p(\cdot)} \|\Delta u_0\|_2.\]
    By \eqref{eqn:a-priori-tm-1} and 
    \eqref{eqn:a-priori-tm-2}, we arrive at the following a priori estimate:
    \begin{equation} \label{eqn:a-priori}
        \sup_{t \in [0,T]} \|u^m(t)\|_{(\alpha)}^2 +
        \int_{0}^{T} \|\partial_t u^m(\tau)\|_2^2 d \tau \leq 3K_1 + \frac{3k(T)}{p^-}K_2.
    \end{equation}

    \textbf{Step 3. Pass to the limits.} 
    It follows from \eqref{eqn:a-priori} and the Aubin-Lions-Simon compactness theorem 
    \cite[Theorem II.5.16]{BOYER2013} that 
    there exist a function $u \in L^{\infty}(0,T;\mathcal{H}_0^2(\Omega)) \cap 
    H^1(0,T;\mathcal{H}_0(\Omega))$ and a subsequence of $\{u^m\}$, which will not be 
    relabeled, such that 
    \begin{subequations}  
        \begin{empheq}{alignat=3}
            &u^m \rightharpoonup u \quad \textrm{weakly-*} && \ \textrm{in} 
            \quad  && L^{\infty}(0,T;\mathcal{H}_0^2(\Omega)), \label{eqn:convergence-1} \\ 
            &u^m \rightharpoonup u \quad \textrm{weakly} && \ \textrm{in}  
            \quad  && H^1(0,T;\mathcal{H}_0(\Omega)), \label{eqn:convergence-2} \\
            &u^m \to u \quad \textrm{strongly} && \ \textrm{in} 
            \quad  && C([0,T];\mathcal{W}_0^{1,p(\cdot)}(\Omega)) \label{eqn:convergence-3} 
        \end{empheq}
    \end{subequations}
    as $m \to \infty$. It follows from \eqref{eqn:convergence-1} and 
    \eqref{eqn:convergence-3} that $C_w([0,T];\mathcal{H}_0^2(\Omega))$.
    Now we show that $u$ is a weak solution to 
    problem \eqref{eqn:main} on $[0,T]$. Fix $n \in \mathbb{N}$. For $m \geq n$,
    multiplying \eqref{eqn:integral-identity-base} by arbitrarily given 
    $\zeta_i(t) \in C^1([0,T])$, summing over $i=1,2,\cdots,n$ and integrating it 
    on $(0,t)$ so that $t \leq T$, we get 
    \begin{align} \label{eqn:integral-identity-base-time}
        &\int_{0}^{t} \big[(\partial_t u^m, \psi) + (\Delta u^m, \Delta \psi) + 
        \alpha ((-\Delta)^s u^m, (-\Delta)^s \psi) \big] dt \notag \\ 
        &= -\int_{0}^{t} \int_{\Omega} \left(1-k(\tau) |\nabla u^m|^{p(x)-2}\right) 
        \nabla u^m \cdot \nabla \psi dxd\tau = -\int_{0}^{t} \int_{\Omega} F(\tau, x, \nabla u^m) 
        \cdot \nabla \psi dxd\tau   
    \end{align}
    for any $\psi \in C^1([0,T];\mathcal{H}_0^2(\Omega))$ of the form 
    $\psi(t,x) = \sum_{j=1}^n \zeta_j(t) \omega_j(x)$. With the help 
    of \eqref{eqn:tartar-inequality}, we deduce from 
    \eqref{eqn:integral-identity-base-time} and 
    \eqref{eqn:convergence-1}-\eqref{eqn:convergence-3} that 
    \begin{equation} \label{eqn:integral-identity-time}
        \int_{0}^{t} \bigg[(\partial_t u, \psi) + (\Delta u, \Delta \psi) + 
        \alpha ((-\Delta)^s u, (-\Delta)^s \psi)  
        + \int_{\Omega} F(\tau, x, \nabla u)   
        \cdot \nabla \psi dx \bigg]d\tau = 0.   
    \end{equation}
    It is easy to see that \eqref{eqn:integral-identity-time} also holds 
    for any $\psi \in L^2(0,T;\mathcal{H}_0^2(\Omega))$  
    by a density argument. The arbitrariness of $t \in (0,T]$ implies that 
    \eqref{eqn:main-weak} holds for any $\varphi \in \mathcal{H}_0^2(\Omega)$ 
    and almost all $t \in (0,T)$. The existence of weak solutions to 
    problem \eqref{eqn:main} is proved. From \eqref{eqn:decreasing-integral-um}, 
    \eqref{eqn:convergence-1}, \eqref{eqn:convergence-2} and weak lower semicontinuity,
    it follows that \eqref{eqn:energy-inequality} holds for a.e. $t \in [0,T]$.
    
    \textbf{Step 4. Uniqueness.} 
    We suppose that $u_1$ and $u_2$ are two weak solutions to problem \eqref{eqn:main} 
    on $[0,T]$, then for any $\varphi \in \mathcal{H}_0^2(\Omega)$ 
    and almost all $t \in (0,T)$, the function $w \coloneqq u_1 - u_2$ satisfies
    \begin{align} \label{eqn:subtract}
        & (\partial_t w, \varphi) +  (\Delta w, \Delta \varphi) + 
        \alpha ((-\Delta)^s w, (-\Delta)^s \varphi) + 
        \int_{\Omega} \nabla w \cdot \nabla \varphi dx \notag \\ 
        &= k(t) \int_{\Omega} \left(|\nabla u_1|^{p(x)-2} \nabla u_1 
        - |\nabla u_2|^{p(x)-2} \nabla u_2\right) \cdot \nabla \varphi dx
    \end{align}
    and $w(0)=0$. Taking $\varphi=w$ in \eqref{eqn:subtract} and integrating 
    over $(0,t)$ for $t \in (0,T]$, we obtain that 
    \begin{equation} \label{eqn:subtract-integral}
        \frac{1}{2}\|w(t)\|_2^2  + 
        \int_{0}^{t} \|w(\tau)\|_{(\alpha)}^2 d\tau 
        = \int_{0}^{t} k(\tau) \int_{\Omega} \left(|\nabla u_1|^{p(x)-2} \nabla u_1 
        - |\nabla u_2|^{p(x)-2} \nabla u_2\right) \cdot \nabla w dx d \tau.
    \end{equation}
    From now until the end of this proof, $C>0$ denotes a 
    constant (may depending on $u_1$ and $u_2$), which may take 
    different values at different places. Set $q(x)=\frac{p(x)}{p(x)-1}$ and 
    $r(x)=\frac{p(x)}{p(x)-2}$. For all $t \in (0,T)$, it follows 
    from H{\"o}lder's inequality, Tartar's type inequality, the 
    Gagliardo-Nirenberg inequality and Young's inequality that 
    \begin{align} \label{eqn:right-side-estimate}
        & k(t)\int_{\Omega} \left(|\nabla u_1|^{p(x)-2} \nabla u_1 - |\nabla u_2|^{p(x)-2} \nabla u_2
        \right) \cdot \nabla w dx \notag \\ 
        & \leq 2k(T) \left\|\left| |\nabla u_1|^{p(x)-2} \nabla u_1 
        - |\nabla u_2|^{p(x)-2} \nabla u_2\right| 
        \right\|_{q(\cdot)} \left\|\nabla w\right\|_{p(\cdot)} \notag \\ 
        & \leq C (p^+-1)^{2-\frac{2}{p^+}} \left\|
        \left(|\nabla u_1|^{p(x)}+|\nabla u_2|^{p(x)}\right)^{\frac{p(x)-2}{p(x)}} 
        |\nabla u_1 - \nabla u_2|\right\|_{q(\cdot)} 
        \left\|\nabla w\right\|_{p(\cdot)} \notag \\ 
        & \leq C \left\|\left(|\nabla u_1|^{p(x)}+|\nabla u_2|^{p(x)}\right)^{\frac{1}{r(x)}}
        \right\|_{r(\cdot)} \left\|\nabla w\right\|_{p(\cdot)}^2 \notag \\ 
        & \leq C \max \left\{ \left(\varrho(|\nabla u_1|) + \varrho(|\nabla u_2|)
        \right)^{\frac{p^--2}{p^-}}, 
        \left(\varrho(|\nabla u_1|) + \varrho(|\nabla u_2|)\right)^{\frac{p^+-2}{p^+}}\right\} 
        \left\|\nabla w\right\|_{p(\cdot)}^2 \notag \\ 
        & \leq C \|w\|_{(\alpha)}^{2\theta} \|w\|_{2}^{2-2\theta} 
        \leq \|w\|_{(\alpha)}^2 + (C \theta^{\theta})^{\frac{1}{1-\theta}} \|w\|_2^2 \notag \\ 
        & \leq \|w\|_{(\alpha)}^2 + C \|w\|_2^2.
    \end{align}
Combining \eqref{eqn:subtract-integral} with \eqref{eqn:right-side-estimate} we know that
\[\|w(t)\|_2^2 \leq C \int_{0}^{t} \|w(\tau)\|_2^2 d \tau\]
for all $t \in (0,T)$. It follows from Gr{\"o}nwall's inequality that $w \equiv 0$ 
a.e. in $(0,T) \times \mathbb{R}^N$ and the proof of the uniqueness of weak solutions 
to problem \eqref{eqn:main} is thereby completed. 
\end{proof}

\begin{remark}
    Definition \ref{def:main-weak} (\ref{enm:main-weak-2}) can be understood as 
    $u(t,\cdot) \to u_0$ strongly in $\mathcal{H}_0(\Omega)$. 
    We can also interpret it as 
    $u(t,\cdot) \rightharpoonup u_0$ weakly in $\mathcal{H}_0^2(\Omega)$ 
    since $u \in C_w([0,T];\mathcal{H}_0^2(\Omega))$.
    Moreover, $u \in L^{\infty}(0,T;\mathcal{H}_0^2(\Omega))$ implies that there exists 
    $N \subset [0,T]$ of measure zero such that for all $t \in [0,T] \backslash N$, 
    $\|u(t)\|_{(\alpha)} < +\infty$. For $t \in N$, we can approximate $t$ by 
    a sequence $\{t_j\} \subset [0,T] \backslash N$. Since 
    $u \in C_w([0,T];\mathcal{H}_0^2(\Omega))$, we have 
    $u(t_j) \rightharpoonup u(t)$ weakly in $\mathcal{H}_0^2(\Omega)$,  
    which implies that $t \in [0,T] \backslash N$ for each there exists 
    $\varXi(t) > 0$ independent of $j$ such that $\|u(t_j)\|_{(\alpha)} \leq \varXi(t)$ 
    for every $j \in \mathbb{N}$. Consequently, $\|u(t)\|_{(\alpha)} \leq 
    \varliminf_{j \to \infty} \|u(t_j)\|_{(\alpha)} \leq \varXi(t) < +\infty$. 
    Therefore, for all $t \in [0,T]$, we conclude that 
    $\|\Delta u(t)\|_{2} \sim \|u(t)\|_{H^2(\mathbb{R}^N)} \sim \|u(t)\|_{(\alpha)} 
    < +\infty$. 
    From the above local existence theorem and the standard continuation procedure, 
    it follows that as long as $\|\Delta u\|_{2}$ remains bounded, 
    the weak solution can be extended. Using this procedure, we can also derive the 
    following alternative property.
\end{remark}

\begin{corollary}
    The weak solution to problem \eqref{eqn:main} either exists globally or 
    blows up in finite time.    
\end{corollary}

\begin{proof}
    To see this we argue by contradiction. Let $u(t)$ be a weak solution to problem 
    \eqref{eqn:main}. Suppose that its maximal existence time 
    $T_*$ is finite but it does not blow up. That is, there exist a $\varTheta>0$ and 
    a sequence $\{t_j\}$ such that $t_j \nearrow T_*$ as $t_j \to \infty$ 
    and $\|u(t_j)\|_{(\alpha)} \leq \varTheta$ for all $j \in \mathbb{N}$. As shown in 
    the proof of the local existence theorem above, there exists 
    $\widehat{T}=\widehat{T}(\varTheta)>0$ 
    such that for each $j \in \mathbb{N}$, $u(t)$ exists on $[t_j,t_j+\widehat{T}]$. 
    When $j$ is sufficiently large, we have $T_* < t_j+\widehat{T}$, which contradicts 
    the assumption that $T_*$ is the maximal existence interval. This completes the proof.
\end{proof}

\section{Global existence and finite time blow-up of solutions with \texorpdfstring{$J(u_0; 0) \leq \underline{d}$}{Lg}} \label{sec4} 

\subsection{The case \texorpdfstring{$J(u_0; 0) < \underline{d}$}{Lg}}

We begin by proving the continuity of $I(u^m(t);t)$ with respect to $t$. 

\begin{lemma} 
    For each $m \in \mathbb{N}$, $I(u^m(t);t) \in C([0,T_m))$, where $u^m(t)$ 
    denotes the approximate solutions constructed in the proof of Theorem 
    \ref{thm:local-well-posedness} and $T_m$ denotes their maximal existence time. 
    Moreover, for $t \in [0,T_m)$, it holds that 
    \begin{equation} \label{eqn:integral-identity-base-2}
    \frac{1}{2} \|u^m(t)\|_2^2 + \int_{0}^{t} I(u^m(\tau);\tau) d \tau 
    = \frac{1}{2} \|u^m(0)\|_2^2.
    \end{equation}
\end{lemma}

\begin{proof}
    Multiplying \eqref{eqn:integral-identity-base} by 
    $g_i^m(t)$ and summing over $i = 1,2,\cdots,m$, we obtain that
    \begin{equation} \label{eqn:nehari-identity-base}
        I(u^m(t);t) = -(u^m(t), \partial_t u^m(t)) 
        = - \frac{1}{2} \frac{d}{dt} \|u^m(t)\|_2^2
    \end{equation}
    for $t \in [0,T_m)$. Fix $\widehat{T} \in (0,T_m)$. Since 
    $u^m \in C^1([0,T_m);\mathcal{H}_0^2(\Omega))$, it follows that
    $u^m \in C([0,\widehat{T}];\mathcal{H}_0(\Omega))$ and 
    $\partial_t u^m \in C_w([0,\widehat{T}];\mathcal{H}_0(\Omega))$.
    For any $t \in [0,\widehat{T}]$, consider a sequence $\{t_j\}$ 
    such that $t_j \to t$ as $j \to \infty$, then 
    $u(t_j) \rightharpoonup u(t)$ weakly in $\mathcal{H}_0(\Omega)$. 
    Therefore, there exists a constant $\varXi > 0$ independent of $j$ 
    such that $\|u(t_j)\|_{2} \leq \varXi$ for each $j \in \mathbb{N}$.
    Then 
    \begin{align*} 
         \left|I(u^m(t);t) - I(u^m(t_j);t_j)\right| 
        & = \left|(u^m(t), \partial_t u^m(t)) - (u^m(t_j), \partial_t u^m(t_j))\right| \\
        & \leq \left|(u^m(t)-u^m(t_j), \partial_t u^m(t_j))\right| + 
        \left|(u^m(t)-u^m(t_j), \partial_t u^m(t))\right| \\
        & \leq \varXi \left\|u^m(t)-u^m(t_j)\right\|_2 
        + \left|(u^m(t)-u^m(t_j), \partial_t u^m(t))\right| \to 0
    \end{align*}
    as $j \to \infty$, which implies that $I(u^m(t);t)$ is continuous at 
    any $t \in [0,\widehat{T}]$. It follows from the arbitrariness of 
    $\widehat{T}$ that $I(u^m(t);t) \in C([0,T_m))$ for each $m \in \mathbb{N}$.
    To derive \eqref{eqn:integral-identity-base-2} one only need to 
    integrate \eqref{eqn:nehari-identity-base} over $(0,t)$ so that $t \in [0,T_m)$.
\end{proof}

Now we present a global well-posedness result for problem \eqref{eqn:main} under the 
condition that $J(u_0; 0) < \underline{d}$. 
Note that this case is only possible when $\underline{d} > 0$.

\begin{theorem} \label{thm:global-well-posedness-low}
    Let $u_0 \in \mathcal{H}_0^2(\Omega)$. If $J(u_0;0)<\underline{d}$ and 
    $I(u_0;0)>0$, then problem \eqref{eqn:main} possesses a unique 
    global weak solution $u(t)$. That is, for any $T' \in (0,+\infty)$, 
    the unique weak solution $u(t)$ exists on $[0,T']$ and there exists a 
    positive constant $M>0$ independent of $u_0$ such that 
    \begin{align} \label{eqn:a-priori-M}
        \sup_{t \in [0,T']} \|u(t)\|_{(\alpha)}^2 
    + \int_{0}^{T'} \|\partial_t u(\tau)\|_2^2 d \tau \leq M.
    \end{align}
    Moreover, there exists a 
    constant $\delta_1>0$ such that $\|u(t)\|_{H^2(\mathbb{R}^N)} = O(e^{-\delta_1 t})$.
\end{theorem}

\begin{proof}
    The proof will be completed in two steps.

    \textbf{Step 1. Global well-posedness.} 

    Since $J(\cdot;0)$ and $I(\cdot;0)$ are continuous on $\mathcal{H}_0^2(\Omega)$, 
    we find $m_* \in \mathbb{N}$ such that 
    $J(u^m(0);0) < \underline{d}$ and $I(u^m(0);0)>0$ for $m \geq m_*$. From now 
    until the end of this proof, we let $m \geq m_*$. The fact that $J(u^m(t);t)$ 
    is nonincreasing with respect to $t \in [0,T_m)$ implies 
    that $J(u^m(t);t) < \underline{d}$ for $t \in (0,T_m)$. 
    
    Now we prove that for 
    $t \in (0,T_m)$, $u^m(t) \in \mathcal{N}_+(t)$. If not, then 
    \[S \coloneqq 
    \left\{t \in (0,T_m) : I(u^m(t);t) \leq 0 \ \textrm{and} \ u^m(t) \not\equiv 0\right\} 
    \ne \varnothing.\]
    Set $t_0 = \inf S$. It follows from $I(u^m(0);0)>0$ and 
    the continuity of $t \mapsto I(u^m(t);t)$ that $t_0 \in (0,T_m)$. 
    A standard proof by contradiction yields that $I(u^m(t_0);t_0)=0$.
    We assert that $u^m(t_0) \not\equiv 0$. Otherwise, by Lemma \ref{lem:0-interior} and 
    the continuity of $u^m(t)$ in $\mathcal{H}_0^2(\Omega)$ we know that there exists  
    $\delta>0$ depending only on $T'$ such that for every
    $t \in (t_0,t_0+\delta]$, either $u^m(t) \equiv 0$ or $I(u^m(t);t) > 0$.
    If the latter is true, then it follows from 
    \begin{align*}
        I(u^m(t);t) &=  \|u^m(t)\|_{(\alpha)}^2 
        - k(t) \int_{\Omega} |\nabla u^m(t)|^{p(x)} dx \\
        &= 2J(u^m(t);t) - k(t) \int_{\Omega} \frac{p(x)-2}{p(x)} 
        |\nabla u^m(t)|^{p(x)} dx > 0 
    \end{align*}
    that $J(u^m(t);t)>0$ for $t \in (t_0,t_0+\delta]$. However, since 
    $J(u^m(t_0);t_0)=0$ and $J(u^m(t);t)$ is nonincreasing with respect to $t$, 
    this leads to a contradiction. Thus, for $t \in [t_0,t_0+\delta]$, it must hold that
    $u^m(t) \equiv 0$. Since $\delta$ depends only on $T'$, this process can be repeated to 
    conclude that $u^m(t) \equiv 0$ for all $t \in [t_0,T_m)$. However, this 
    contradicts $S \ne \varnothing$. Therefore, $u^m(t_0) \equiv 0$ cannot hold, leaving
    $u^m(t_0) \in \mathcal{N}(t_0)$ as the only possibility. By the definition 
    of $\mathcal{N}(t_0)$, we have
    \[J(u^m(t_0);t_0) \geq d(t_0) \geq \underline{d},\]
    which contradicts $J(u^m(t_0);t_0) < \underline{d}$. 

    The argument above shows that our initial assumption is false. Thus, we have proved 
    that for $t \in (0,T_m)$, $u^m(t) \in \mathcal{N}_+(t)$. It follows from 
    \eqref{eqn:integral-identity-base-2} that
    \begin{equation} \label{eqn:coefficient-norm-control}
        \sum_{j=1}^{m} (g_j^m(t))^2 = \|u^m(t)\|_2^2 = \|u^m(0)\|_2^2
    - 2 \int_{0}^{t} I(u^m(\tau);\tau) d \tau \leq \|u_0\|_2^2
    \end{equation}
    for $t \in (0,T_m)$. Thus, $T_m = T'$. Then, using the following inequality 
    \begin{equation} \label{eqn:energy-norm-control}
        J(u^m(t);t) \geq \frac{p^--2}{2p^-} \|u^m(t)\|_{(\alpha)}^2 + 
    \frac{1}{p^-} I(u^m(t);t) \geq \frac{p^--2}{2p^-} \|u^m(t)\|_{(\alpha)}^2
    \end{equation}
    and \eqref{eqn:decreasing-integral-um}, we have 
    \[\frac{p^--2}{2p^-} \|u^m(t)\|_{(\alpha)}^2 + 
    \int_{0}^{t} \|\partial_t u^m(\tau)\|_2^2 d \tau \leq J(u^m(0);0) < \underline{d}\]
    for $t \in [0,T']$, which implies that 
    \[ \sup_{t \in [0,T']} \|u^m(t)\|_{(\alpha)}^2 
    + \int_{0}^{T'} \|\partial_t u^m(\tau)\|_2^2 d \tau 
    \leq \frac{2p^-\underline{d}}{p^--2} + \underline{d} 
    = \frac{3p^--2}{p^--2}\underline{d}.\]    
    Once again, following the lines of Step 3 in the proof of 
    Theorem \ref{thm:local-well-posedness},
    we deduce that there exists $u \in L^{\infty}(0,T';\mathcal{H}_0^2(\Omega)) \cap 
    H^1(0,T';\mathcal{H}_0(\Omega))$ and a subsequence of $\{u^m\}$ (still denoted by 
    $\{u^m\}$) such that 
    \begin{subequations}  
        \begin{empheq}{alignat=3}
            &u^m \rightharpoonup u \quad \textrm{weakly-*} && \ \textrm{in} 
            \quad  && L^{\infty}(0,T';\mathcal{H}_0^2(\Omega)), \label{eqn:convergence2-1} \\ 
            &u^m \rightharpoonup u \quad \textrm{weakly} && \ \textrm{in}  
            \quad  && H^1(0,T';\mathcal{H}_0(\Omega)), \\
            &u^m \to u \quad \textrm{strongly} && \ \textrm{in} 
            \quad  && C([0,T'];\mathcal{W}_0^{1,p(\cdot)}(\Omega)) \label{eqn:convergence2-3}
        \end{empheq}
    \end{subequations}
    as $m \to \infty$ and $u$ is a weak solution to problem \eqref{eqn:main}.
    Moreover, $u \in C_w([0,T'];\mathcal{H}_0^2(\Omega))$ and there exists a positive 
    constant $M>0$ independent of $m$ and $u_0$ such that 
    \eqref{eqn:a-priori-M} holds. 
    Step 4 in the proof of Theorem \ref{thm:local-well-posedness} shows that 
    the weak solution is always unique within its existence interval. The arbitrariness of
    $T'$ thereby completes the proof of global well-posedness.

    \textbf{Step 2. Decay rate.} 

    For a given $t \geq 0$, we first assume that $u^m(t) \not\equiv 0$. 
    By a following a process similar to that in the proof of Lemma \ref{lem:depth-bound} (iii), 
    we find that there exists $\mu^*(t) \in (1, +\infty)$ such 
    that $\mu^*(t)u^m(t) \in \mathcal{N}(t)$. Thus, 
    \begin{align*}
        0 = I(\mu^*(t)u^m(t);t) & \leq \left(\mu^*(t)\right)^2 \|u^m(t)\|_{(\alpha)}^2 
        - k(t) \left(\mu^*(t)\right)^{p^-} \int_{\Omega} |\nabla u^m(t)|^{p(x)} dx \\
        &= \left(\left(\mu^*(t)\right)^2-\left(\mu^*(t)\right)^{p^-}\right) 
        \|u^m(t)\|_{(\alpha)}^2 + \left(\mu^*(t)\right)^{p^-} I(u^m(t);t),
    \end{align*}
    that is
    \[I(u^m(t);t) \geq \left(1-\left(\mu^*(t)\right)^{2-p^-}\right) \|u^m(t)\|_{(\alpha)}^2.\]   
    On the other hand,
    \begin{align*}
        \underline{d} \leq d(t) & \leq J(\mu^*(t)u^m(t);t) \\
        & \leq \frac{1}{2} \left(\mu^*(t)\right)^2 \|u^m(t)\|_{(\alpha)}^2 
        - k(t) \left(\mu^*(t)\right)^{p^-} \int_{\Omega} \frac{1}{p(x)} 
        |\nabla u^m(t)|^{p(x)} dx \\ 
        & \leq \frac{1}{2} \left(\mu^*(t)\right)^{2} \|u^m(t)\|_{(\alpha)}^2 
        - k(t) \left(\mu^*(t)\right)^2 \int_{\Omega} \frac{1}{p(x)} 
        |\nabla u^m(t)|^{p(x)} dx \\ 
        & = \left(\mu^*(t)\right)^2 J(u^m(t);t) 
        \leq \left(\mu^*(t)\right)^2 J(u^m(0);0), 
    \end{align*}
    which implies that there exist a constant 
    \[\delta_0 \coloneqq \sup_{m \geq m^*} \left(\frac{J(u^m(0);0)}{\underline{d}}\right)
    ^{\frac{p^--2}{2}} \in (0,1)\] 
    such that 
    \[\left(\mu^*(t)\right)^{2-p^-} \leq \left(\frac{J(u^m(0);0)}{\underline{d}}\right)
    ^{\frac{p^--2}{2}} \leq \delta_0, \quad \forall m \geq m_*.\]
    Therefore, we have
    \begin{equation} \label{eqn:nehari-lower}
        I(u^m(t);t) \geq (1-\delta_0) \|u^m(t)\|_{(\alpha)}^2
    \end{equation}
    for $t \geq 0$, and it also holds for the case where $u^m(t) \equiv 0$. 
    
    Using the following inequality 
    \[J(u^m(t);t) \leq \frac{p^+-2}{2p^+} \|u^m(t)\|_{(\alpha)}^2 + 
    \frac{1}{p^+} I(u^m(t);t)\] 
    and \eqref{eqn:nehari-lower}, we obtain that 
    \begin{equation} \label{eqn:energy-nehari-control}
        J(u^m(t);t) \leq \frac{p^+-2 \delta_0}{2p^+(1-\delta_0)} I(u^m(t);t).
    \end{equation}
    Integrating both sides of the above inequality over $(t,t')$, where 
    $t'$ can be chosen sufficiently large, we get 
    \begin{align*}
        \int_{t}^{t'}J(u^m(\tau);\tau) d \tau 
        & \leq \frac{p^+-2 \delta_0}{2p^+(1-\delta_0)} 
        \int_{t}^{t'}I(u^m(\tau);\tau) d \tau \\
        & \leq \frac{p^+-2 \delta_0}{2p^+(1-\delta_0)} \|u^m(t)\|_{2}^2 \\
        & \leq \frac{p^+-2 \delta_0}{2B_2^2p^+(1-\delta_0)} \|u^m(t)\|_{(\alpha)}^2 \\
        & \leq \frac{p^-(p^+-2 \delta_0)}{B_2^2p^+(p^--2)(1-\delta_0)} J(u^m(t);t)
    \end{align*}
    by using \eqref{eqn:integral-identity-base-2}, \eqref{eqn:energy-norm-control} 
    and \eqref{eqn:energy-nehari-control}. Here 
    \[B_2 \coloneqq \inf_{v \in \mathcal{H}_0^2(\Omega) \backslash \{0\}} 
    \frac{\|v\|_{(\alpha)}}{\|v\|_{2}}.\] 
    Furthermore, it follows from $J(u^m(t); t) \leq \underline{d}$ that for any $t'>t$,
    $\int_{t}^{t'} J(u^m(\tau); \tau)  d\tau$ has a uniform positive upper 
    bound. Since $J(u^m(t); t)$ is nonnegative, 
    $\int_{t}^{+\infty} J(u^m(\tau); \tau)  d\tau$ converges and it holds that 
    \begin{equation} \label{eqn:energy-integral-control}
        \int_{t}^{+\infty} J(u^m(\tau); \tau)  d\tau \leq \frac{1}{2 \delta_1} J(u^m(t);t),
    \end{equation}
    where $\delta_1 \coloneqq \frac{B_2^2p^+(p^--2)(1-\delta_0)}{2p^-(p^+-2 \delta_0)}$. 
    From \cite[Lemma 1]{MARTINEZ1999251}, it follows that
    \[J(u^m(\tau); \tau) \leq J(u^m(0); 0) e^{1-2 \delta_1 t} 
    < \underline{d} e^{1-2 \delta_1 t},\]
    combining with \eqref{eqn:energy-norm-control}, we deduce that for any $t \geq 0$, 
    \[\|u^m(t)\|_{(\alpha)} \leq \left(\frac{2ep^-\underline{d}}{p^--2}\right)^{\frac{1}{2}} 
    e^{-\delta_1 t}.\]
    Then, \eqref{eqn:convergence2-1} implies that for any $t > 0$, 
    \[\operatorname*{ess\,sup}_{\tau \in (t,t+1)} \|u(\tau)\|_{_{H^2(\mathbb{R}^N)}}
    \leq \varliminf_{m \to \infty} \operatorname*{ess\,sup}_{\tau \in (t,t+1)} 
    \|u^m(\tau)\|_{_{H^2(\mathbb{R}^N)}} \lesssim e^{-\delta_1 t}.\]
    Therefore, there exists $N_t \subset (t,t+1)$ of measure zero such that  
    $\|u(\tau)\|_{_{H^2(\mathbb{R}^N)}} \lesssim e^{-\delta_1 t}$
    for $\tau \in (t,t+1) \backslash N_t$. Consider a sequence $\{t_j\} 
    \subset (t,t+1) \backslash N_t$ 
    such that $t_j \to t$ as $j \to \infty$. By weak continuity, we have 
    \[\|u(t)\|_{_{H^2(\mathbb{R}^N)}} \leq \varliminf_{j \to \infty} 
    \|u(t_j)\|_{_{H^2(\mathbb{R}^N)}} \lesssim e^{-\delta_1 t}.\]
    The arbitrariness of $t > 0$ concludes the proof.
\end{proof}

We then give a result concerning the finite time blow-up properties of solutions 
to problem \eqref{eqn:main} when $J(u_0;0)<\underline{d}$.

\begin{theorem} \label{thm:blow-up-low-energy}
    Assume that $u_0 \in \mathcal{H}_0^2(\Omega)$. If $J(u_0;0)<\underline{d}$ and 
    $I(u_0;0)<0$, then any solution to problem \eqref{eqn:main} blows up in finite time. 
\end{theorem}

\begin{proof}
    Some ideas of this proof comes from Theorem 3.2 in \cite{KBIRIALAOUI20141723}.
    Assume that $u(t)$ is a global solution to problem \eqref{eqn:main}. 
    Firstly we prove that for 
    $t \geq 0$, $u(t) \in \mathcal{N}_-(t)$. If not, the continuity of 
    $t \mapsto I(u(t);t)$ implies that there exists $t_1>0$ such that 
    $I(u(t_1);t_1)=0$ and $I(u(t);t)<0$ for $t \in [0,t_1)$. 
    From the continuity of $t \mapsto \|u(t)\|_{(\alpha)}^2$ and the proof of 
    Lemma \ref{lem:away-from-0}, it follows that 
    \[\|u(t_1)\|_{(\alpha)}^2 = \lim_{t \nearrow t_1} \|u(t_1)\|_{(\alpha)}^2
    \geq \min \left\{ 
        \left(k(t_1)S_{p(\cdot)}^{p^-}\right)^{\frac{2}{2-p^-}}, 
        \left(k(t_1)S_{p(\cdot)}^{p^+}\right)^{\frac{2}{2-p^+}}
    \right\} > 0,\]
    which implies that $u(t_1) \in \mathcal{N}(t_1)$. By the definition of $d(t_1)$ 
    we have $J(u(t_1);t_1) \geq d(t_1) \geq \underline{d} > J(u_0;0)$.
    Since $J(u(t);t)$ is nonincreasing with respect to $t$, 
    this leads to a contradiction. Thus, $u(t) \in \mathcal{N}_-(t)$ for $t \geq 0$.

    Let 
    \[F_1(t) = \frac{1}{2} \|u(t)\|_2^2,\]
    then $F_1(t)$ is absolutely continuous and 
    \[F_1'(t) = (u,u_t) = -I(u(t);t) = k(t) \int_{\Omega} \frac{p(x)-2}{p(x)} 
        |\nabla u(t)|^{p(x)} dx-2J(u(t);t)\]
    for almost all $t \geq 0$. In the case where $\underline{d}=0$, since
    $J(u(t);t) <\underline{d}=0$, we can choose $C_0=k(0)\frac{p^--2}{p^-}$ 
    such that for almost all $t \geq 0$,
    \begin{equation} \label{eqn:differential-modular}
        F_1'(t) \geq C_0 \int_{\Omega} |\nabla u(t)|^{p(x)} dx.
    \end{equation}
    If $\underline{d}>0$, from \eqref{eqn:unstable-depth-upper}, we have 
    \begin{align*}
        F_1'(t) & \geq k(t) \left(1-\frac{J(u(t);t)}{d(t)}\right)
        \int_{\Omega} \frac{p(x)-2}{p(x)} |\nabla u(t)|^{p(x)} dx \\
        & \geq k(0) \frac{p^--2}{p^-} \left(1-\frac{J(u_0;0)}{\underline{d}}\right) 
        \int_{\Omega} |\nabla u(t)|^{p(x)} dx
    \end{align*}
    for a.e. $t \geq 0$. Thus, we can choose 
    \begin{equation*}
        C_0 \coloneqq \left\{
        \begin{array}{ll}
            k(0) \frac{p^--2}{p^-} \left(1-\frac{J(u_0;0)}{\underline{d}}\right), 
             & \textrm{if} \quad \underline{d} > 0,  \\ 
            k(0) \frac{p^--2}{p^-}, & \textrm{if} \quad \underline{d} = 0,  
        \end{array}
        \right. 
    \end{equation*}
    such that \eqref{eqn:differential-modular} holds for a.e. $t \geq 0$.
    Define the sets $\Omega_+=\{x \in \Omega: |\nabla u| \geq 1\}$ and 
    $\Omega_-=\{x \in \Omega: |\nabla u| < 1\}$. Then for almost all $t \geq 0$, 
    \[F_1'(t) \geq \int_{\Omega_+} |\nabla u(t)|^{p^-} dx 
    \geq C_0 |\Omega|^{\frac{2-p^-}{2}} \left(\int_{\Omega_+} 
    |\nabla u(t)|^{2} dx\right)^{\frac{p^-}{2}}\] 
    and 
    \[F_1'(t) \geq \int_{\Omega_-} |\nabla u(t)|^{p^+} dx 
    \geq C_0 |\Omega|^{\frac{2-p^+}{2}} \left(\int_{\Omega_-} 
    |\nabla u(t)|^{2} dx\right)^{\frac{p^+}{2}},\] 
    which implies that there exists a positive constant
    $C_1 \coloneqq \lambda_1 \min \left\{C_0^{\frac{2}{p^+}} |\Omega|^{\frac{2-p^+}{p^+}}, 
    C_0^{\frac{2}{p^-}} |\Omega|^{\frac{2-p^-}{p^-}} \right\}$ 
    such that 
    \[(F_1'(t))^{\frac{2}{p^-}} 
    \left(1 + (F_1'(t))^{\frac{2}{p^+} - \frac{2}{p^-}}\right)=  
    (F_1'(t))^{\frac{2}{p^-}}   + (F_1'(t))^{\frac{2}{p^+}}\geq C_1 F_1(t)\]
    for a.e. $t \geq 0$. For almost all $t \geq 0$, $F_1'(t) = -I(u(t);t) >0$ holds,  
    which implies that 
    \[F_1(t) \geq F_1(0) + \int_{0}^{t} F_1'(\tau) d \tau \geq F_1(0) > 0\] 
    for all $t \geq 0$. 
    Therefore, for a.e. $t \geq 0$, $F_1'(t) \geq C_2 \coloneqq \min \left\{
        \left(\frac{C_1}{2}\|u_0\|_2^2\right)^{\frac{p^+}{2}},
        \left(\frac{C_1}{2}\|u_0\|_2^2\right)^{\frac{p^-}{2}}\right\}.$ Set 
    \[C_3 \coloneqq \left(\frac{C_1}{1+C_2^{\frac{2}{p^+} - \frac{2}{p^-}}}\right)
    ^{\frac{2}{p^-}},\] 
    then for a.e. $t \geq 0$, 
    \begin{equation} \label{eqn:differential-inequality}
        F_1'(t) \geq C_3 (F_1(t))^{\frac{p^-}{2}}.
    \end{equation}
    Integrating \eqref{eqn:differential-inequality} over $(0,t)$, we deduce that for 
    all $t \geq 0$,
    \[\|u(t)\|_2^2 \geq \frac{2}{\left[\left(\frac{1}{2}\|u_0\|_2^2\right)
    ^{1-\frac{p^-}{2}} - \frac{p^--2}{2}C_3 t\right]^{\frac{2}{p^--2}}}.\]
    Hence, 
    \[\lim_{t \to \frac{2}{C_3(p^--2)} \left(\frac{1}{2}\|u_0\|_2^2\right)
    ^{1-\frac{p^-}{2}}} \|u(t)\|_2^2 = + \infty,\]
    which contradicts the assumption that $u(t)$ is a global solution. 
    Therefore, $u(t)$ blows up in finite time and its maximal existence interval
    $T_*$ satisfies 
    \[T_* \leq \frac{2}{C_3(p^--2)} \left(\frac{1}{2}\|u_0\|_2^2\right)
    ^{1-\frac{p^-}{2}}.\]
    This fact concludes the proof.
\end{proof}

\subsection{The case \texorpdfstring{$J(u_0; 0) = \underline{d}$}{Lg}}

By combining Theorem \ref{thm:global-well-posedness-low} with an approximation argument, 
we obtain a global existence result for the case $J(u_0; 0) = \underline{d}$.

\begin{theorem}
    Let $u_0 \in \mathcal{H}_0^2(\Omega)$. If $J(u_0;0)=\underline{d}$ and 
    $I(u_0;0)>0$, then problem \eqref{eqn:main} possesses a unique 
    global weak solution $u(t)$.
\end{theorem}

\begin{proof} 
    Set $\mu_n = \frac{n}{n+1}$, $u_{0n}(x)=\mu_n u_0(x)$, $n \in \mathbb{N}$. 
    Consider the following 
    problem 
    \begin{equation}
        \left\{
        \begin{array}{ll}
            u_t + \mathcal{L}_{\alpha} u =
            \operatorname{div} \left(F(t, x, \nabla u)\right), 
             & x \in \Omega, \ t > 0,  \\ 
            u(0,x) = u_{0n}(x), & x \in \Omega,  \\
            u = 0, & x \in \mathbb{R}^N \backslash \Omega. 
        \end{array}
        \right. \label{eqn:main-critical}
    \end{equation}
    When $\mu \in (0,1)$, it follows from $I(u_0;0)>0$ that 
    \begin{align*}
        I(\mu u_0;0) &\geq \mu^2 \|u_0\|_{(\alpha)}^2 
        - k(0) \mu^{p^-} \int_{\Omega} |\nabla u_0|^{p(x)} dx \\
        &= \mu^{p^-} I(u_0;0) +\left(\mu^2-\mu^{p^-}\right) \|u_0\|_{(\alpha)}^2 > 0. 
    \end{align*}
    Set $j_0(\mu)=J(\mu u_0;0)$. Then for $\mu \in (0,1)$, it holds that 
    \begin{equation*}
        \frac{d}{d \mu}j_0(\mu) = \mu \|u_0\|_{(\alpha)}^2 
        - k(0) \int_{\Omega} \mu^{p(x)-1} |\nabla u_0|^{p(x)} dx 
        = \frac{1}{\mu} I(\mu u_0;0) > 0,
    \end{equation*}
    that is, $\mu \mapsto J(\mu u_0;0)$ is strictly increasing on $(0,1)$. Then from 
    $\mu_n \in (0,1)$ we obtain that $I(u_{0n};0)>0$ and 
    $J(u_{0n};0)<J(u_0;0)=\underline{d}$ for any $n \in \mathbb{N}$. By Theorem 
    \ref{thm:global-well-posedness-low} we know that for each $n$ and $T'>0$, problem
    \eqref{eqn:main-critical} possesses a unique weak solution $u_n$ on $[0,T']$ 
    satisfying 
    \begin{equation}
        (\partial_t u_n, \varphi) + (\Delta u_n, \Delta \varphi) +
        \alpha ((-\Delta)^{s} u_n, (-\Delta)^{s} \varphi) + 
        \int_{\Omega} F(t, x, \nabla u_n) \cdot \nabla \varphi d x
        = 0
        \label{eqn:main-weak-m}
    \end{equation}
    for any $\varphi \in \mathcal{H}_0^2(\Omega)$ and almost all $t \in (0,T')$. 
    Besides, there exists a positive constant $M>0$ independent of $n$ such that 
    \begin{align} \label{eqn:a-priori-m-M}
        \sup_{t \in [0,T']} \|u_n(t)\|_{(\alpha)}^2 
    + \int_{0}^{T'} \|\partial_t u_n(\tau)\|_2^2 d \tau \leq M.
    \end{align}
    Therefore, there exists $u \in L^{\infty}(0,T';\mathcal{H}_0^2(\Omega)) \cap 
    H^1(0,T';\mathcal{H}_0(\Omega))$ and a subsequence of $\{u_n\}$ (still denoted by 
    $\{u_n\}$) such that 
    \begin{subequations}  
        \begin{empheq}{alignat=3}
            &u_n \rightharpoonup u \quad \textrm{weakly-*} && \ \textrm{in} 
            \quad  && L^{\infty}(0,T';\mathcal{H}_0^2(\Omega)), \notag \\ 
            &u_n \rightharpoonup u \quad \textrm{weakly} && \ \textrm{in}  
            \quad  && H^1(0,T';\mathcal{H}_0(\Omega)), \notag \\
            &u_n \to u \quad \textrm{strongly} && \ \textrm{in} 
            \quad  && C([0,T'];\mathcal{W}_0^{1,p(\cdot)}(\Omega))  \notag 
        \end{empheq}
    \end{subequations}
    as $n \to \infty$. Letting $n \to \infty$ in \eqref{eqn:main-weak-m} we 
    get 
    \begin{equation*}
        (u_t, \varphi) + (\Delta u, \Delta \varphi) +
        \alpha ((-\Delta)^{s} u, (-\Delta)^{s} \varphi) + 
        \int_{\Omega} F(t, x, \nabla u) \cdot \nabla \varphi d x
        = 0
    \end{equation*}
    for any $\varphi \in \mathcal{H}_0^2(\Omega)$ and almost all $t \in (0,T')$, 
    which implies that $u(t)$ is the weak solution to problem \eqref{eqn:main-weak} 
    on $[0,T']$. The arbitrariness of $T'$ thereby completes the proof. 
\end{proof}

We deduce the following finite time blow-up result for $J(u_0; 0) = \underline{d}$ from 
Theorem \ref{thm:blow-up-low-energy}.

\begin{theorem}
    Assume that $u_0 \in \mathcal{H}_0^2(\Omega)$. If $J(u_0;0)=\underline{d}$ and 
    $I(u_0;0)<0$, then any solution to problem \eqref{eqn:main} blows up in finite time.
\end{theorem}

\begin{proof}
    Assume that $u(t)$ is a global solution to problem \eqref{eqn:main}. Since 
    $J(u_0;0)=\underline{d} > 0$ and $I(u_0;0)<0$,
    the continuity of $J(u(t);t)$ and $I(u(t);t)$ with respect to $t$ implies that
    there exists $t_1 > 0$ such that $J(u(t);t)>0$ and $I(u(t);t)<0$ for 
    $t \in [0,t_1)$. 
    From $(u,u_t) = -I(u(t);t) > 0$ for a.e. $t \in [0,t_1]$, we have that 
    $u_t$ is not identically zero on $[0,t_1]$. Hence, 
    \[J(u(t_1);t_1) \leq J(u_0;0) - \int_{0}^{t_1} \|\partial_t u(\tau)\|_2^2 d \tau 
    < \underline{d}.\]
    Working as in the proof of Theorem \ref{thm:blow-up-low-energy} we get 
    $u(t) \in \mathcal{N}_-(t)$ for $t \geq t_1$. Furthermore, we can choose 
    \[C_0=k(0)\frac{p^--2}{p^-} \left(1-\frac{J(u(t_1);t_1)}{\underline{d}}\right)\]
    such that for almost all $t \geq t_1$, 
    \begin{align*}
        F_1'(t) & \geq k(t) \left(1-\frac{J(u(t);t)}{d(t)}\right)
        \int_{\Omega} \frac{p(x)-2}{p(x)} |\nabla u(t)|^{p(x)} dx \\
        & \geq k(t) \left(1-\frac{J(u(t_1);t_1)}{\underline{d}}\right)
        \int_{\Omega} \frac{p(x)-2}{p(x)} |\nabla u(t)|^{p(x)} dx \\
        & = C_0
        \int_{\Omega} |\nabla u(t)|^{p(x)} dx.
    \end{align*} 
    The remaining part of the proof follows the same arguments as that of
    Theorem \ref{thm:blow-up-low-energy}, and thus we omit the details here.
\end{proof}

\section{High initial energy case \texorpdfstring{$J(u_0; 0) > 0$}{Lg}} \label{sec5}

Let us recall the properties of $\lambdabar_{\varsigma}(t)$ presented in 
Section \ref{sec2}. Following the ideas of Gazzola et al. \cite{GAZZOLA2005961}, 
we first establish a sufficient condition for the global existence of solutions 
when $J(u_0; 0) > \overline{d}$.

\begin{theorem} \label{thm:global-well-posedness-high}
    Let $k(t)$ be bounded and $u_0 \in \mathcal{N}_+(0)$. Assume that  
    $\overline{d}<J(u_0;0) \leq \widetilde{d}$ with 
    $\widetilde{d} \in (\overline{d},+\infty)$. Then, there exists 
    $\mathfrak{L} : \mathcal{H}_0^2(\Omega) \to \mathbb{R}_{>0}$
    such that if $\|u_0\|_2^2 \leq \mathfrak{L}(u_0)$ holds, 
    problem \eqref{eqn:main} possesses a unique 
    global weak solution $u(t)$.
\end{theorem}

\begin{proof}
    Recall the approximate solution $u^m$ defined in the proof of 
    Theorem \ref{thm:local-well-posedness} and its maximal existence interval $T_m$. 
    We find $m_{**}=m_{**}(u_0) \in \mathbb{N}$ 
    such that $J(u^m(0);0) >\overline{d}$, $I(u^m(0);0)>0$ for $m \geq m_{**}$. 
    From now until the end of this proof, we let $m \geq m_{**}$. Setting
    \[K_3 = K_3(u_0) \coloneqq \sup_{m \geq m_{**}} J(u^m(0);0),\]
    we define 
    $\mathfrak{L}(u_0) \coloneqq \underline{\lambdabar}_{K_3(u_0)}$. 
    The boundedness of $k(t)$ ensures that $\mathfrak{L}(u_0)>0$. Assume that
    $\|u_0\|_2^2 \leq \mathfrak{L}(u_0)$ holds, then
    \[\|u^m(0)\|_2^2 \leq \|u_0\|_2^2 \leq \underline{\lambdabar}_{K_3}.\]
    We claim that for 
    $t \in (0,T_m)$, $u^m(t) \in \mathcal{N}_+(t)$. Otherwise, by a similar argument as in 
    the proof of Theorem \ref{thm:global-well-posedness-low}, we know that 
    there exists $t_0 \in (0,T_m)$ such that $u^m(t) \in \mathcal{N}_+(t)$ 
    for $t \in [0,t_0)$ and $u^m(t_0) \in \mathcal{N}(t)$. 
    Then \eqref{eqn:nehari-identity-base} implies 
    that $\partial_t u^m \not \equiv 0$ for $(t,x) \in (0,t_0) \times \mathbb{R}^N$ and 
    \begin{equation} \label{eqn:2-norm-upper}
        \|u^m(t_0)\|_2^2 < \|u^m(0)\|_2^2 \leq \underline{\lambdabar}_{K_3}.
    \end{equation}
    However, from \eqref{eqn:decreasing-integral-um}, it follows that 
    $J(u^m(t_0);t_0) < J(u^m(0);0) \leq K_3$, which implies 
    that $u^m(t_0) \leq \mathcal{N}_{K_3}(t_0)$. This yields that 
    \[\|u^m(t_0)\|_2^2 \geq \lambdabar_{K_3}(t_0) \geq \underline{\lambdabar}_{K_3},\]
    which contradicts \eqref{eqn:2-norm-upper}. Therefore, 
    for $t \in [0,T_m)$, $u^m(t) \in \mathcal{N}_+(t)$. This 
    implies that \eqref{eqn:coefficient-norm-control} holds for $t \in (0,T_m)$ 
    so that $T_m = T'$. Similar to the proof of 
    Theorem \ref{thm:global-well-posedness-low}, we can obtain the following estimate: 
    \[\sup_{t \in [0,T']} \|u^m(t)\|_{(\alpha)}^2 
    + \int_{0}^{T'} \|\partial_t u^m(\tau)\|_2^2 d \tau 
    \leq \frac{3p^--2}{p^--2}\widetilde{d},\]   
    and once again derive \eqref{eqn:convergence2-1}-\eqref{eqn:convergence2-3}. 
    The remaining proof is the same as that of 
    Theorem \ref{thm:global-well-posedness-low} and hence we omit it.
\end{proof}

Further assumptions on $p$ lead to a more explicit sufficient condition, 
as shown in the following corollary.

\begin{corollary}
    Let $p^+ \leq \frac{2(N+4)}{N+2}$ and $k(t)$ be bounded. 
    Let $u_0 \in \mathcal{H}_0^2(\Omega)$. Assume that  
    $\overline{d}<J(u_0;0) \leq \widetilde{d}$ with 
    $\widetilde{d} \in (\overline{d},+\infty)$. If 
    $\|u_0\|_2^2 < \underline{\lambdabar}_{\infty}$, 
    then problem \eqref{eqn:main} possesses a unique 
    global weak solution $u(t)$.
\end{corollary}

\begin{proof}
    It follows from Lemma \ref{lem:2-norm-inf-lower}(ii) that 
    $\underline{\lambdabar}_{\infty}>0$. 
    We claim that $u_0 \in \mathcal{N}_+(0)$. By the definition 
    of $\underline{\lambdabar}_{\infty}(0)$, we know that 
    \begin{equation} \label{eqn:2-norm-upper-initial}
        \|u_0\|_2^2 < \underline{\lambdabar}_{\infty} \leq \lambdabar_{\infty}(0) 
        = \inf_{v \in \mathcal{N}(0)} \|v\|_2^2,
    \end{equation}
    from which it follows that $u_0 \notin \mathcal{N}(0)$. 
    If $u_0 \in \mathcal{N}_-(0)$, a similar argument as in the proof of 
    Lemma \ref{lem:depth-bound}(iii) shows that there exists a constant $\mu^*\in (0,1)$ 
    such that $\mu^* u_0 \in \mathcal{N}(0)$. Review the definition 
    of $\lambdabar_{\infty}(0)$, we obtain that 
    \[\|u_0\|_2^2 \geq (\mu^*)^{-2} \lambdabar_{\infty}(0) > \lambdabar_{\infty}(0),\]
    which contradicts \eqref{eqn:2-norm-upper-initial}. Thus, 
    $u_0 \in \mathcal{N}_+(0)$ and it holds that 
    \[\|u_0\|_2^2 < \underline{\lambdabar}_{\infty} \leq 
    \underline{\lambdabar}_{K_3}(u_0) = \mathfrak{L}(u_0).\]
    The conclusion then follows 
    directly from Theorem \ref{thm:global-well-posedness-high}.
\end{proof}

Due to the time-dependence of $\mathcal{N}(t)$, the method used in \cite{GAZZOLA2005961} is 
not readily applicable for analyzing the $\omega$-limit set of $u_0$, which makes it difficult 
to establish a blow-up result for the case $J(u_0; 0) > \overline{d}$ by contradiction. 
Thanks to Levine’s concavity argument \cite{LEVINE1973371}, we are able to obtain a 
sufficient condition for finite time blow-up with arbitrarily high initial 
energy $J(u_0; 0) > 0$. 

\begin{lemma}[\cite{LEVINE1973371}] \label{lem:levine-concavity}
    Assume that $F: [0,T] \to \mathbb{R}_{>0}$ is differentiable on $[0,T]$, 
    and $F'$ is absolutely continuous on $[0,T]$ with $F'(0)>0$. If there exists 
    positive constant $\alpha>0$ such that $FF''-(1+\alpha)F' \geq 0$ a.e. on $[0,T]$, 
    then $T \leq t^* \coloneqq \frac{F(0)}{\alpha F'(0)}$, and there exists 
    $t_* \in (0, t^*]$ such that $F(t) \to +\infty$ as $t \nearrow t_*$. 
\end{lemma}

\begin{theorem} \label{thm:blow-up-high-energy}
    Assume that $u_0 \in \mathcal{H}_0^2(\Omega)$. If 
    $0<J(u_0;0)<\frac{p^--2}{2p^-}B_2^2\|u_0\|_{2}^2$, then any solution to 
    problem \eqref{eqn:main} blows up in finite time. 
\end{theorem}

\begin{proof}
    Assume that $u(t)$ is a global solution to problem \eqref{eqn:main}.  
    It follows from
    \begin{align} \label{eqn:nehari-norm-control}
        I(u(t);t) &\leq p^- J(u(t);t) - \frac{p^--2}{2} \|u(t)\|_{(\alpha)}^2 \notag \\
        &\leq p^- J(u_0;0) - \frac{p^--2}{2} B_2^2 \|u(t)\|_2^2 \notag \\
        &< \frac{p^--2}{2} B_2^2 \left(\|u_0\|_2^2 - \|u(t)\|_2^2\right)
    \end{align} 
    that $I(u_0;0) < 0$. By employing an approach similar to the one 
    in the proof of Theorem \ref{thm:blow-up-low-energy}, we prove that 
    for $t \geq 0$, $u(t) \in \mathcal{N}_-(t)$. 
    If not, then there exists $t_1>0$ such that 
    $I(u(t_1);t_1)=0$ and $I(u(t);t)<0$ for $t \in [0,t_1)$. Recall the function 
    $F_1(t)$ as defined in the proof of Theorem 
    \ref{thm:blow-up-low-energy}. From 
    \[F_1(t_1) \geq F_1(0) + \int_{0}^{t_1} F_1'(\tau) d \tau \geq F_1(0) > 0,\]
    it follows that for $t \in [0,t_1)$, 
    \begin{align} \label{eqn:norm-nondecreasing}
    0<J(u_0;0)<\frac{p^--2}{2p^-}B_2^2\|u_0\|_{2}^2 \leq 
    \frac{p^--2}{2p^-}B_2^2\|u(t_1)\|_{2}^2.
    \end{align} 
    On the other hand, from \eqref{eqn:nehari-norm-control} we have
    \[J(u_0;0) \geq J(u(t_1);t_1) \geq \frac{1}{p^-} I(u(t_1);t_1) 
    + \frac{p^--2}{2p^-} \|u(t_1)\|_{(\alpha)}^2 \geq
    \frac{p^--2}{2p^-} B_2^2 \|u(t_1)\|_{(\alpha)}^2,\]
    which contradicts \eqref{eqn:norm-nondecreasing}. 
    Therefore, for $t \geq 0$, $u(t) \in \mathcal{N}_-(t)$ and $F_1(t)$ is strictly 
    increasing. 

    Set $\eta \coloneqq \frac{p^-}{p^--1} 
    \left(\frac{p^--2}{2p^-}B_2^2\|u_0\|_{2}^2 - J(u_0;0)\right)$ and 
    $\sigma \coloneqq \frac{2\|u_0\|_{2}^2}{(p^--2)\eta}$. For any $T'>0$, we define 
    \[F_2(t) = \frac{1}{2} \int_{0}^{t} \|u(\tau)\|_2^2 d \tau 
    + \frac{T' - t}{2} \|u_0\|_2^2  + \frac{\eta}{2} (t + \sigma)^2\]
    for $t \in [0,T']$. Then $F_2(t)$ is differentiable on $[0,T']$, and for
    all $t \in [0,T']$,
    \[F_2'(t) = \frac{1}{2} \|u(t)\|_2^2 - \frac{1}{2} \|u_0\|_2^2 
    + \eta (t + \sigma) = F_1(t) - F_1(0) + \eta (t + \sigma) > 0,\]
    which implies that $F_2(t) > 0$ for all $t \in [0,T']$.
    Moreover, $F_2'(t)$ is absolutely continuous on $[0,T']$ and 
    \[F_2''(t) = (u,u_t) + \eta = -I(u(t);t) + \eta = F_1'(t) + \eta\]
    for almost all $t \in [0,T']$. From \eqref{eqn:nehari-norm-control}, it follows that 
    for a.e. $t \in [0,T']$, 
    \begin{align} \label{eqn:F2-dd-lower}
        F_2''(t) & \geq -p^- J(u(t);t) 
        + \frac{p^--2}{2} B_2^2 \|u(t)\|_{2}^2 + \eta \notag \\
        & \geq -p^- J(u_0;0) + p^- \int_{0}^{t} \|\partial_t u(\tau)\|_2^2 d \tau  
        + \frac{p^--2}{2} B_2^2 \|u_0\|_{2}^2 + \eta.
    \end{align} 
    By the Cauchy-Schwarz inequality and H{\"o}lder's inequality
    we know that for a.e. $t \in [0,T']$,
    \[\varPhi(t) \coloneqq \left(\int_{0}^{t} \|u(\tau)\|_2^2 d \tau 
    + \eta (t + \sigma)^2 \right) 
    \left(\int_{0}^{t} \|\partial_t u(\tau)\|_2^2 d \tau + \eta \right) - 
    \left(\int_{0}^{t}(u,\partial_t u(\tau)) d \tau + \eta (t + \sigma) \right)^2 \geq 0.\]
    Then for a.e. $t \in [0,T']$, it holds that
    \begin{align} \label{eqn:F2-d-upper}
        (F_2'(t))^2 & = \left(\int_{0}^{t}(u,\partial_t u(\tau)) d \tau 
        + \eta (t + \sigma)\right)^2 \notag \\
        & = \left(2F_2(t)-(T'-t) \|u_0\|_2^2 \right) 
        \left(\int_{0}^{t} \|\partial_t u(\tau)\|_2^2 d \tau + \eta \right) -
        \varPhi(t) \notag \\
        & \leq 2 F_2(t) \left(\int_{0}^{t} \|\partial_t u(\tau)\|_2^2 d \tau + \eta \right).
    \end{align}
    In view of \eqref{eqn:F2-dd-lower} and \eqref{eqn:F2-d-upper}, we deduce that 
    \begin{equation*}
        F_2(t) F_2''(t) - \frac{p^-}{2} (F_2'(t))^2 
        \geq F_2(t) \left(\frac{p^--2}{2} B_2^2 \|u_0\|_{2}^2 - p^- J(u_0;0) - 
        (p^- - 1) \eta \right) = 0
    \end{equation*}
    for a.e. $t \in [0,T']$. It follows from Lemma \ref{lem:levine-concavity} that 
    \[T' \leq \frac{2F_2(0)}{(p^--2)F_2'(0)} 
    = \frac{T' \|u_0\|_{2}^2 + \eta \sigma^2}{(p^--2)\eta \sigma} 
    = \frac{T'}{2} + \frac{2\|u_0\|_{2}^2}{(p^--2)^2\eta},\]
    namely, 
    \[T' \leq \frac{4\|u_0\|_{2}^2}{(p^--2)^2\eta} 
    = \frac{8(p^--1)\|u_0\|_{2}^2}{(p^--2)^2
    \left((p^--2)B_2^2 \|u_0\|_{2}^2 - 2p^- J(u_0;0)\right)},\]
    and there exists $t_*$ such that
    \[\lim_{t \nearrow t_*} 
    \int_{0}^{t} \|u(\tau)\|_2^2 d \tau = + \infty,\]
    which contradicts the assumption that $u(t)$ is a global solution. 
    Thus, $u(t)$ blows up in finite time and its maximal existence interval
    $T_*$ satisfies 
    \[T_* \leq \frac{8(p^--1)\|u_0\|_{2}^2}{(p^--2)^2
    \left((p^--2)B_2^2 \|u_0\|_{2}^2 - 2p^- J(u_0;0)\right)},\]
    which concludes the proof.
\end{proof}

The section ends with a corollary of Theorem \ref{thm:blow-up-high-energy}.

\begin{corollary}
    For any $R>0$, there exists a $u_R \in \mathcal{H}_0^2(\Omega)$ such that 
    $J(u_R;0) = R$ and the solution of problem \eqref{eqn:main} 
    blows up in finite time with initial data $u_R$.
\end{corollary}

\begin{proof}
    Let $\Omega_1$ and $\Omega_2$ be two disjoint open subdomains of $\Omega$. 
    Consider an arbitrary $v_0 \in \mathcal{H}_0^2(\Omega_1) \backslash \{0\}$, 
    one can choose $\overline{\mu} > 0$ large enough such that 
    $J(\overline{\mu} v_0;0) \leq 0$ and 
    $\|\overline{\mu} v_0\|_2 > B_2^{-2}\frac{2p^-}{p^--2} R$.
    On the other hand, for fixed $\overline{\mu} > 0$, we choose 
    $w_0 \in \mathcal{H}_0^2(\Omega_2)$ such 
    that $J(w_0;0)+J(\overline{\mu} v_0;0) = R$. Then 
    $u_R \coloneqq \overline{\mu} v_0 + w_0$ satisfies $J(u_R;0) = R$, and 
    it follows that 
    \[\frac{p^--2}{2p^-}B_2^2 \|u_R\|_2 \geq
    \frac{p^--2}{2p^-}B_2^2\|\overline{\mu} v_0\|_2 > R = J(u_R;0) > 0.\]
    Thus the proof is completed by Theorem \ref{thm:blow-up-high-energy}.
\end{proof}

\section{Lower bound for the lifespan} \label{sec6}

In this section, we shall establish a lower bound for the blow-up time. 
The following lemma shows that when $p^+ < \frac{2(N+4)}{N+2}$, blow-up in the 
$H^2(\mathbb{R}^N)$-norm is equivalent to blow-up in the $L^2(\mathbb{R}^N)$-norm. 

\begin{lemma} \label{lem:2-norm-blowup}
    Assume that $p^+ < \frac{2(N+4)}{N+2}$. Let $u(t)$ be a solution to 
    problem \eqref{eqn:main} with maximal existence time $T_* < + \infty$. Then 
    the blow up also occurs in the $L^2(\mathbb{R}^N)$-norm, that is, 
    $\|u(t)\|_2 \to + \infty$ as $t \nearrow T_*$.
\end{lemma}

\begin{proof}
    For $t \in [0,T_*)$, it follows from $J(u(t);t) \leq J(u_0;0)$ that 
    \begin{align} \label{eqn:energy-decreasing}
    \frac{1}{2} \|u(t)\|_{(\alpha)}^2 & \leq 
    J(u_0;0) + \frac{k(t)}{p^-} \int_{\Omega} |\nabla u|^{p(x)} dx \notag \\ & \leq 
    J(u_0;0) + \frac{k(T_*)}{p^-} |\Omega| + \frac{k(T_*)}{p^-} \|\nabla u\|_{p^+}^{p^+}.
    \end{align}
    Using the Gagliardo-Nirenberg 
    inequality and Young's inequality, we know that
    there exists a positive constant $\widetilde{C}_3$ such that 
    \begin{equation} \label{eqn:gagliardo-3}
        \|\nabla u\|_{p^+}^{p^+} 
        \leq \widetilde{C}_3 \|u\|_{(\alpha)}^{\theta p^+} \|u\|_2^{(1-\theta)p^+} 
        \leq \widetilde{C}_3 \left(\varepsilon^{\frac{2}{\theta p^+}} \|u\|_{(\alpha)}^2 
        + \varepsilon^{-\frac{2}{2-\theta p^+}} \|u\|_2^{\frac{2(1-\theta)p^+}{2-\theta p^+}}\right)
    \end{equation}
    holds for all $\varepsilon>0$,
    where $\theta p^+ = \frac{(N+2)p^+-2N}{4} < 2$. 
    Choosing $\varepsilon = \left(\frac{4 \widetilde{C}_3 k(T_*)}{p^-}\right)
    ^{-\frac{\theta p^+}{2}}$ in \eqref{eqn:gagliardo-3} and combining 
    it with \eqref{eqn:energy-decreasing}, we obtain that 
    \[\|u(t)\|_{H^2(\mathbb{R}^N)}^2 \sim \|u(t)\|_{(\alpha)}^2 
    \leq 4J(u_0;0)+ \frac{4k(T_*)}{p^-} |\Omega| 
    + \left(\frac{4 \widetilde{C}_3 k(T_*)}{p^-}\right)^{\frac{2}{2-\theta p^+}} 
    \left(\|u\|_2^2\right)^{1+\frac{p^+-2}{2-\theta p^+}}.\]
    Letting $t \nearrow T_*$ concludes the proof.
\end{proof}

We now proceed to estimate a lower bound for the lifespan.

\begin{theorem}
    Suppose that $u(t)$ is a solution to 
    problem \eqref{eqn:main} with maximal existence time $T_* < + \infty$. 
    If $p^+ < \frac{2(N+4)}{N+2}$, then there holds 
    \[T_* \geq \int_{F_1(0)}^{+\infty} \frac{dy}{C_4 y^{r^+} + C_5 y^{r^-}},\]
    where $F_1(0) = \frac{1}{2} \|u_0\|_2^2$, and $r^-$, $r^+$, $C_4$, $C_5$ 
    are positive constants independent of $u(t)$ that will be defined later.
\end{theorem}

\begin{proof}
    Recall the function $F_1(t)$ defined in the proof of Theorem \ref{thm:blow-up-low-energy}.
    A simple computation shows that for a.e. $t \in (0,T_*)$,
    \begin{align} \label{eqn:n-nehari-upper}
        F_1'(t) &= - I(u(t)) = - \|u\|_{(\alpha)}^2 
        + k(t) \int_{\Omega} |\nabla u|^{p(x)} dx \notag \\
        & \leq - \|u\|_{(\alpha)}^2 + \kappa^* \|\nabla u\|_{p^+}^{p^+} 
        + \kappa^* \|\nabla u\|_{p^-}^{p^-},
    \end{align}
    where $\kappa^*$ denotes an arbitrary upper bound for $k(T_*)$. 
    Choosing $\varepsilon = \left(2 \kappa^* \widetilde{C}_3\right)
    ^{-\frac{\theta p^+}{2}}$ in \eqref{eqn:gagliardo-3}, we deduce that for 
    all $t \in [0,T_*)$, 
    \begin{equation} \label{eqn:pp-norm-upper}
        \kappa^* \|\nabla u\|_{p^+}^{p^+} 
        \leq \frac{1}{2} \|u\|_{(\alpha)}^2 + 
        \left(2^{\frac{p^+}{2}} \kappa^* \widetilde{C}_3\right)^{\frac{2}{2-\theta p^+}} 
        \left(\frac{1}{2} \|u\|_2^2\right)^{\frac{(1-\theta)p^+}{2-\theta p^+}} 
        = \frac{1}{2} \|u\|_{(\alpha)}^2 + C_4 (F_1(t))^{r^+},
    \end{equation}
    where 
    $C_4 \coloneqq \left(2^{\frac{p^+}{2}} \kappa^* \widetilde{C}_3\right)^{\frac{2}{2-\theta p^+}} = 
    \left(2^{\frac{p^+}{2}} \kappa^* \widetilde{C}_3\right)^{\frac{4}{2N+8-(N+2)p^+}}$, 
    $r^+ \coloneqq \frac{(1-\theta)p^+}{2-\theta p^+} 
    = \frac{2N-(N-2)p^+}{2N+8-(N+2)p^+}$. 
    Similarly, utilizing the Gagliardo-Nirenberg inequality and Young's inequality again, 
    we obtain that there exists a constant $\widetilde{C}_4 > 0$ such that 
    \begin{equation} \label{eqn:gagliardo-4}
        \|\nabla u\|_{p^-}^{p^-} 
        \leq \widetilde{C}_4 \|u\|_{(\alpha)}^{\gamma p^-} \|u\|_2^{(1-\gamma)p^-} 
        \leq \widetilde{C}_4 \left(\varepsilon^{\frac{2}{\gamma p^-}} \|u\|_{(\alpha)}^2 
        + \varepsilon^{-\frac{2}{2-\gamma p^-}} \|u\|_2^{\frac{2(1-\gamma)p^-}{2-\gamma p^-}}\right)
    \end{equation} 
    for all $\varepsilon>0$ and $t \in [0,T_*)$, 
    where $\gamma p^- = \frac{(N+2)p^--2N}{4} < 2$. 
    Choosing $\varepsilon = \left(2 \kappa^* \widetilde{C}_4\right)
    ^{-\frac{\gamma p^-}{2}}$ in \eqref{eqn:gagliardo-4}, we get
    \begin{equation} \label{eqn:pm-norm-upper}
        \kappa^* \|\nabla u\|_{p^-}^{p^-} \leq \frac{1}{2} \|u\|_{(\alpha)}^2 + C_5 (F_1(t))^{r^-},
    \end{equation}
    where 
    $C_5 \coloneqq \left(2^{\frac{p^-}{2}} \kappa^* \widetilde{C}_4\right)^{\frac{4}{2N+8-(N+2)p^-}}$, 
    $r^- \coloneqq \frac{2N-(N-2)p^-}{2N+8-(N+2)p^-}$. 
    Combining \eqref{eqn:n-nehari-upper}, \eqref{eqn:pp-norm-upper}, and \eqref{eqn:pm-norm-upper}, 
    we know that for a.e. $t \in (0,T_*)$,
    \begin{equation} \label{eqn:f1-di-upper}
        F_1'(t) \leq C_4 (F_1(t))^{r^+} + C_5 (F_1(t))^{r^-}.
    \end{equation}
    Integrating \eqref{eqn:f1-di-upper} from $0$ to $t$, it follows that 
    \[\int_{F_1(0)}^{F_1(t)} \frac{dy}{C_4 y^{r^+} + C_5 y^{r^-}} \leq t\]
    holds for all $t \in [0,T_*)$.
    Letting $t \nearrow T_*$, Lemma \ref{lem:2-norm-blowup} implies that 
    $F_1(t) \to +\infty$, which leads to 
    \[T_* \geq \int_{F_1(0)}^{+\infty} \frac{dy}{C_4 y^{r^+} + C_5 y^{r^-}}.\]
    It is easy to check that the right-hand side of the above inequality is convergent, 
    as $2<p^- \leq p^+ < \frac{2(N+4)}{N+2}$ implies $r^+ \geq r^- > 1$. 
    This fact concludes the proof.
\end{proof} 

\section{Numerical experiments} \label{sec7}

We begin this section by performing a numerical discretization of \eqref{eqn:main} 
on the two-dimensional rectangular domain
$\Omega=(0,L_x) \times (0,L_y)$. This discretization is based on finite difference method and 
Fourier spectral method, partially inspired by the work of A. Bueno-Orovio et 
al. \cite{BUENOOROVIO2014}. In the $x$ and $y$ directions, we select $N_x$ and $N_y$ nodes 
respectively, which results in spatial grid sizes 
of $\Delta x = L_x / (N_x+1)$, $\Delta y = L_y / (N_y+1)$.
Denote by $\Delta t$ the time step size, then the spatio-temporal grid is given by
\[\begin{array}{ll}
    t_n = n \Delta t, & n = 0,1,2,\cdots \\
    x_i = i\Delta x, & y_j = j\Delta y, \quad i = 0,\cdots,N_x+1, \quad j = 0,\cdots,N_y+1.
\end{array}\]
For any function $v(t,x,y)$, we denote by $v^n = \{v_{i,j}^n\}_{\forall i, j}$ the
discrete function defined on the above grid, which approximates the value of 
$v(t_n,x_i,y_j)$. The following notations and assumptions are introduced for 
the formulation of the numerical scheme:
\begin{align*}
    &x_{-1} = -\Delta x, \ y_{-1} = -\Delta y, \ 
    x_{N_x+2} = L_x + \Delta x, \ y_{N_y+2} = L_y + \Delta y, \\
    &u^n_{i,-1} = u^n_{i,0} = u^n_{i,N_y+1} = u^n_{i,N_y+2} = 0, \\
    &u^n_{-1,j} = u^n_{0,j} = u^n_{N_x+1,j} = u^n_{N_x+2,j} = 0, \\
    &u^0_{i,j} = u_0(x_i,y_j), \ k^n = k(t_n), 
    \ p_{i,j} = p(x_i,y_j), \\
    &\mathcal{A}u = \left((-\Delta)^2 + \alpha (-\Delta)^{2s} -\Delta\right) u, \
    \mathcal{B}u = \operatorname*{div}\left(|\nabla u|^{p(x,y)-2}\nabla u\right).
\end{align*}
We intend to use a semi-implicit scheme formally expressed as 
\begin{equation} \label{eqn:semi-implicit-form}
    \frac{u^{n+1} - u^n}{\Delta t} + (\mathcal{A} u)^{n+1} = - k^n (\mathcal{B}u)^n.
\end{equation}
To this end, we begin by discretizing $(\mathcal{B}u)^n$.
For $i = 0,\cdots,N_x+1, j = 0,\cdots,N_y+1$, we define a grid function $c^n$ given by 
\[c^n_{i,j} = \left(\left(\frac{u^n_{i + 1,j} - u^n_{i - 1,j}}{2 \Delta x}\right)^2
+\left(\frac{u^n_{i, j + 1} - u^n_{i,j - 1}}{2 \Delta y}\right)^2\right)
^{\frac{p_{i,j}}{2}-1}.\]
For $i = 1,\cdots,N_x, j = 1,\cdots,N_y$, set
\[c^n_{i+\frac{1}{2},j} = \frac{c^n_{i+1,j} + c^n_{i,j}}{2}, \
c^n_{i-\frac{1}{2},j} = \frac{c^n_{i,j} + c^n_{i-1,j}}{2}, \ 
c^n_{i,j+\frac{1}{2}} = \frac{c^n_{i,j+1} + c^n_{i,j}}{2}, \
c^n_{i,j-\frac{1}{2}} = \frac{c^n_{i,j} + c^n_{i,j-1}}{2},\]
then $(\mathcal{B}u)^n$ can be discretized using a standard central difference approach:
\begin{align}
    (\mathcal{B}u)_{i,j}^n 
    &= \frac{c^n_{i+\frac{1}{2},j}\frac{u^n_{i+1,j} - u^n_{i,j}}{\Delta x} 
    - c^n_{i-\frac{1}{2},j}\frac{u^n_{i,j} - u^n_{i-1,j}}{\Delta x}}{\Delta x}
    + \frac{c^n_{i,j+\frac{1}{2}}\frac{u^n_{i,j+1} - u^n_{i,j}}{\Delta y}
     - c^n_{i,j-\frac{1}{2}}\frac{u^n_{i,j} - u^n_{i,j-1}}{\Delta y}}{\Delta y} \notag \\
    &= \frac{1}{2(\Delta x)^2} \left[(c^n_{i+1,j}+c^n_{i,j}) u^n_{i+1,j} 
    - (c^n_{i+1,j} + 2c^n_{i,j} + c^n_{i-1,j}) u^n_{i,j} 
    + (c^n_{i,j}+c^n_{i-1,j}) u^n_{i-1,j}\right] \notag \\ 
    & \quad + \frac{1}{2(\Delta y)^2} \left[(c^n_{i,j+1}+c^n_{i,j}) u^n_{i,j+1} 
    - (c^n_{i,j+1} + 2c^n_{i,j} + c^n_{i,j-1}) u^n_{i,j} 
    + (c^n_{i,j}+c^n_{i,j-1}) u^n_{i,j-1}\right].
\end{align}
To handle the second term on the left-hand side of \eqref{eqn:semi-implicit-form}, 
we use the two-dimensional discrete sine transform (DST). For a function 
$v=\{v_{i,j}\}_{\forall i, j}$ defined on the aforementioned spatial grid, 
its DST is given by
\[\widehat{v}(\omega_1,\omega_2) = \sum_{i=1}^{N_x} \sum_{j=1}^{N_y} v_{i,j} 
\sin (\omega_1 x_i) \sin (\omega_2 y_j),\]
and it can be reconstructed by the two-dimensional inverse discrete sine transform (IDST):
\[v_{i,j} = \frac{4}{(N_x+1)(N_y+1)} 
\sum_{m=1}^{N_x} \sum_{l=1}^{N_y} \widehat{v}(\omega_1,\omega_2) 
\sin (\omega_1 x_i) \sin (\omega_2 y_j),\]
where 
\[\omega_1 = \frac{\pi m}{L_x}, \ \omega_2 = \frac{\pi l}{L_y}, \ 
m = 1, \cdots,N_x, \ n = 1, \cdots,N_y.\]
Let $\lambda(\omega_1,\omega_2) = \omega_1^2+\omega_2^2$. An approximation 
of the DST to $(\mathcal{A} u)^{n+1}$ can then be obtained by
\[\widehat{(\mathcal{A} u)^{n+1}} = \left(\lambda^2(\omega_1,\omega_2) 
+ \lambda^{2s}(\omega_1,\omega_2) + \lambda(\omega_1,\omega_2)\right) 
\widehat{u}^{n+1}(\omega_1,\omega_2).\]
Consequently, equation \eqref{eqn:semi-implicit-form} can be computed by 
\[\widehat{u}^{n+1}(\omega_1,\omega_2) = 
\frac{\widehat{u}^{n}(\omega_1,\omega_2)-\Delta t \cdot k^n 
\widehat{(\mathcal{B}u)^{n}}(\omega_1,\omega_2)}
{1+\Delta t \left[\left(\omega_1^2+\omega_2^2\right)^2 
+ \left(\omega_1^2+\omega_2^2\right)^{2s} + \omega_1^2+\omega_2^2\right]}.\]
Taking the IDST on both sides of the above equality yields $u^{n+1}$.

\begin{figure}[htbp]
    \centering
    \includegraphics[width=\linewidth]{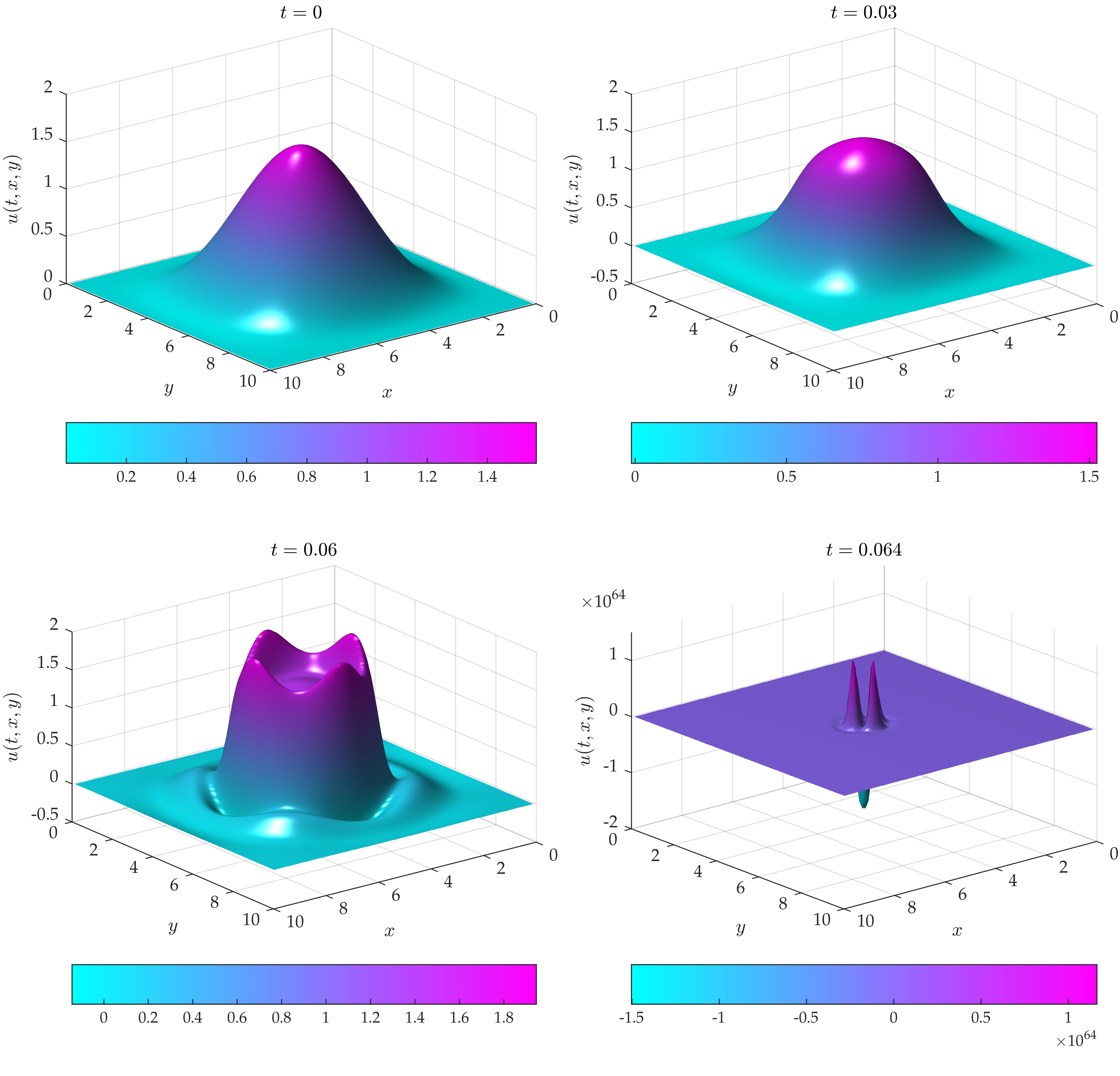}
    \caption{Time evolution of the numerical solution in Example \ref{exm:example1}. 
    Blow-up occurs at approximately $t=0.064$.} \label{fig:example1}
\end{figure}

\begin{figure}[htbp]
    \centering
    \subfigure[Choice of $p(x,y)$]
    {
        \centering
        \includegraphics[width=.5\linewidth]{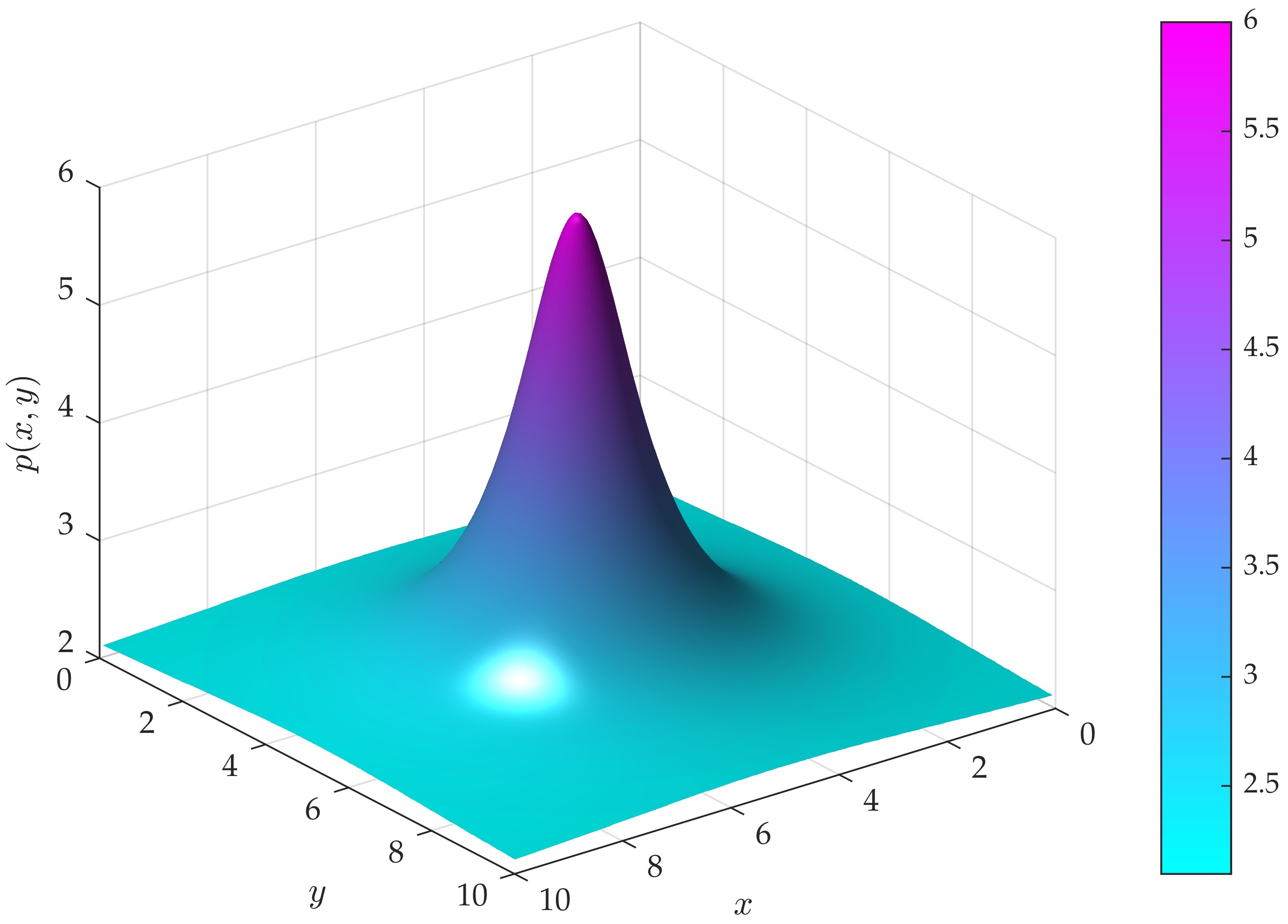} \label{fig:example1-px}
    } 
    \subfigure[Time evolution of $F_1(t)$]
    {
        \centering
        \includegraphics[width=.46\linewidth]{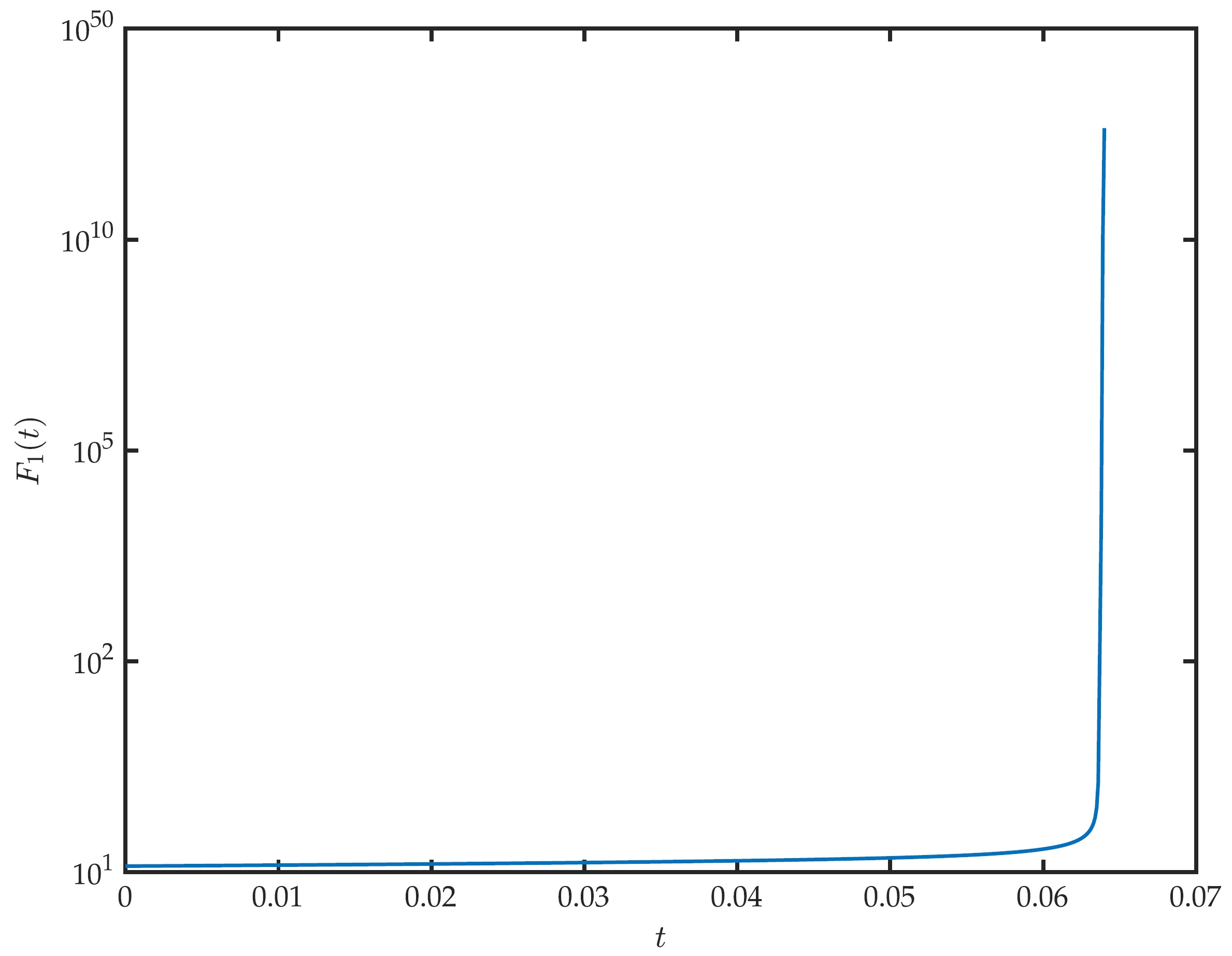} \label{fig:example1-f1t}
    }
    \caption{The choice of the variable exponent and 
    the evolution of $F_1(t) = \frac{1}{2} \|u(t)\|_2^2$ 
    in time, both corresponding to Example \ref{exm:example1}.} \label{fig:example1-px-f1t}
\end{figure}

\begin{figure}[htbp]
    \centering
    \includegraphics[width=\linewidth]{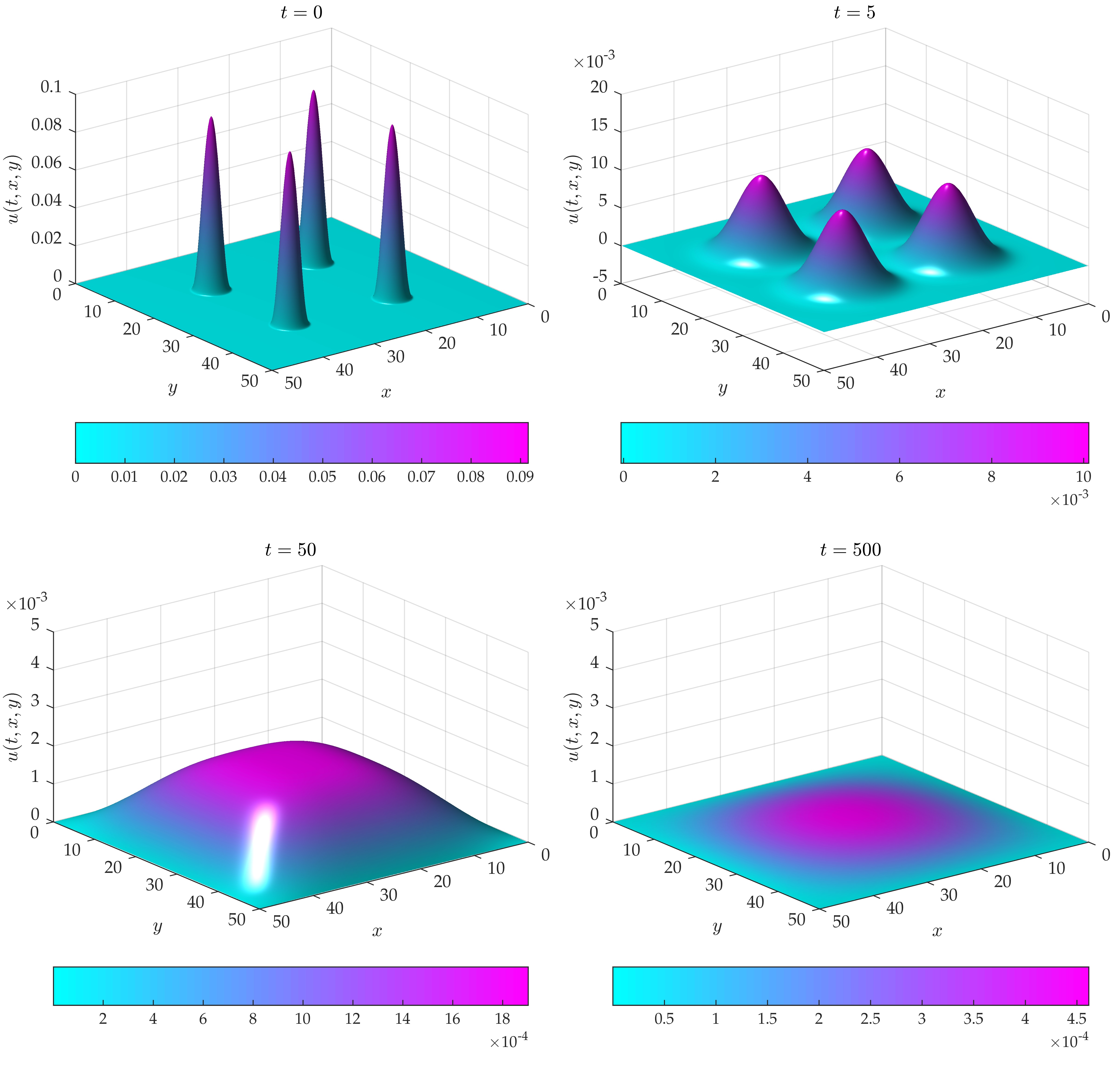}
    \caption{Time evolution of the numerical solution in Example \ref{exm:example2}.} 
    \label{fig:example2}
\end{figure}

\begin{figure}[htbp]
    \centering
    \subfigure[Choice of $p(x,y)$]
    {
        \centering
        \includegraphics[width=.5\linewidth]{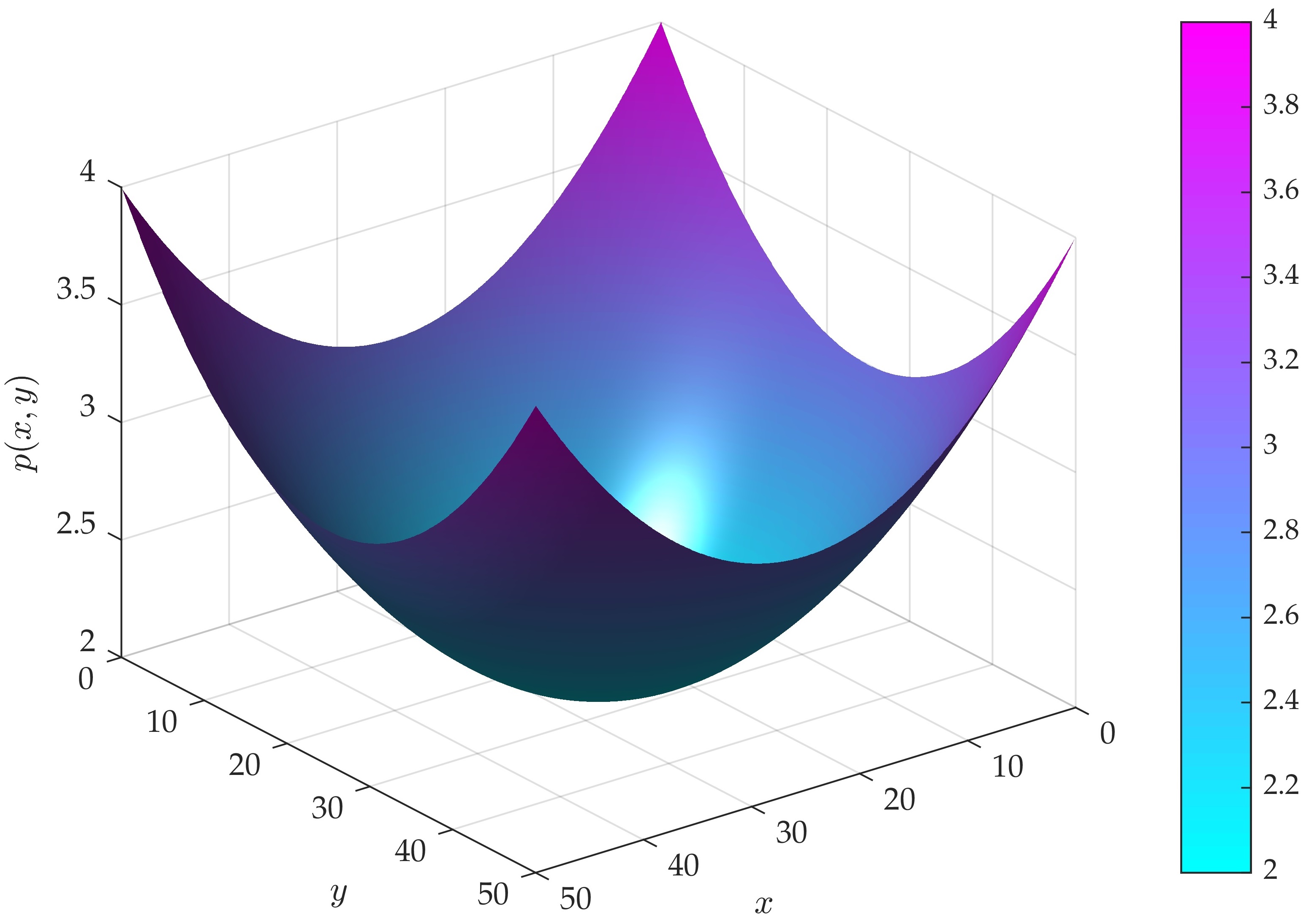} \label{fig:example2-px}
    } 
    \subfigure[Time evolution of $F_1(t)$]
    {
        \centering
        \includegraphics[width=.46\linewidth]{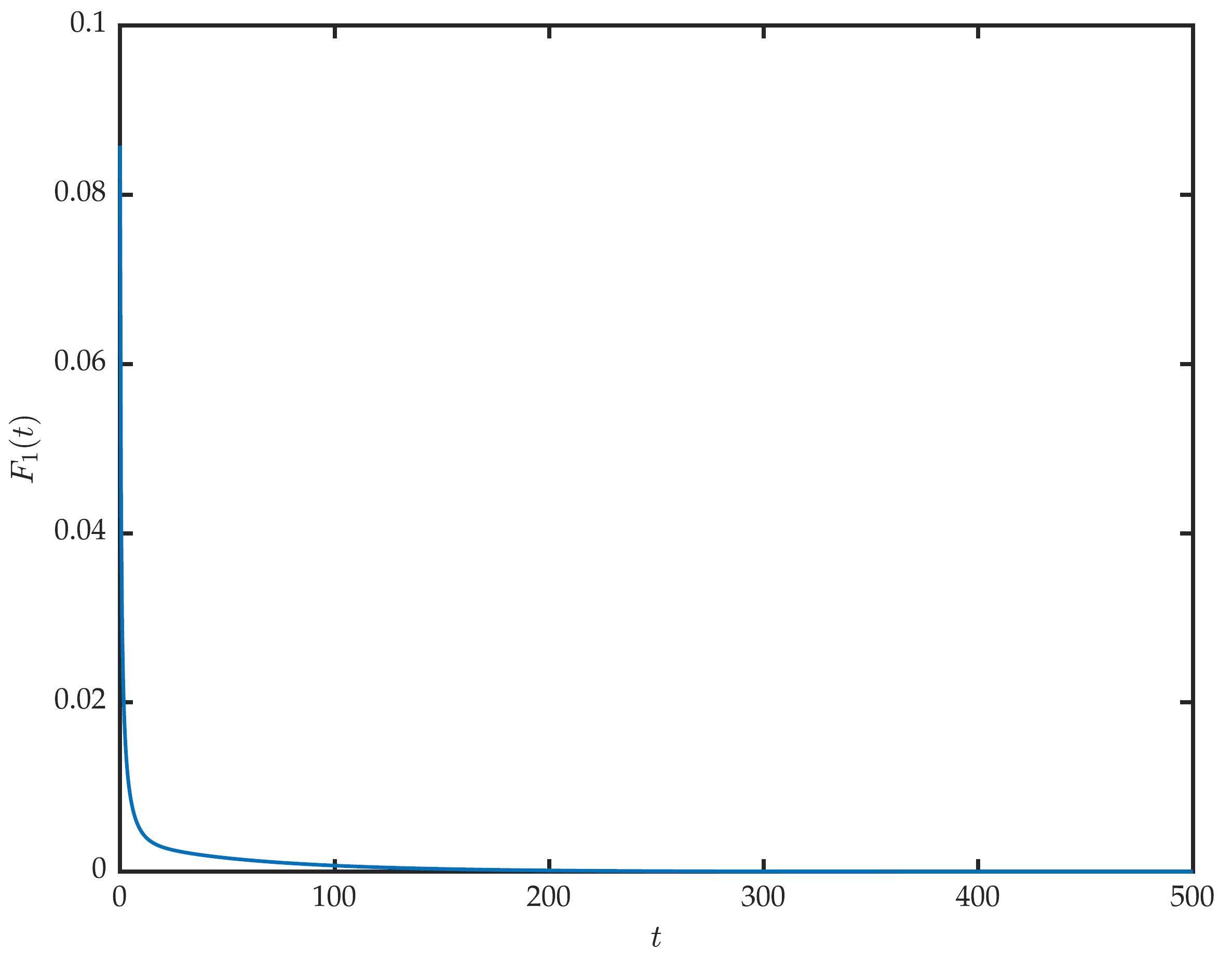} \label{fig:example2-f1t}
    }
    \caption{The choice of the variable exponent and 
    the evolution of $F_1(t) = \frac{1}{2} \|u(t)\|_2^2$ 
    in time, both corresponding to Example \ref{exm:example2}.} 
\end{figure}

Next, we use the above discretization to show the blow-up and large time behavior of 
problem \eqref{eqn:main} through two numerical experiments.

\begin{example} \label{exm:example1}
    Let us consider the numerical solution to problem \eqref{eqn:main} 
    on $\Omega=(0,10)^2$ under the initial condition 
    \[u_0(x,y) = \frac{xy}{400}(10-x)(10-y)\sin^2 \frac{\pi x}{10}\sin^2 \frac{\pi y}{10}.\]
    In this example, we choose 
    \[p(x,y)=2+\frac{5}{(x-5)^2+(y-5)^2+1.5}, \quad k(t)=10 e^t,\]
    and set $N_x = N_y = 150$, $\alpha = -0.95$, $s=0.9$, $\Delta t = 10^{-4}$. 
    Under these settings, the initial values of the energy functional and the Nehari functional 
    can be computed as $J(u_0;0) \approx -13.7952<0$ and $I(u_0;0) \approx -39.5642<0$, 
    respectively. This indicates that $u_0$ satisfies the sufficient condition 
    for blow-up to occur.
    Figure \ref{fig:example1} presents surface plots of $u_0$ as well as the numerical 
    solutions at $t=0.03$, $t=0.06$ and $t=0.064$. Figure \ref{fig:example1-px}
    provides a surface plot displaying the chosen function $p(x,y)$. In Fig. \ref{fig:example1-f1t}, 
    the time evolution of $F_1(t) = \frac{1}{2} \int_{\Omega} |u(t,x,y)|^2 dx$ is shown. 
    It can be observed that the $L^2$-norm of the numerical solution increases rapidly for 
    $t \geq 0.06$, and blow-up occurs at approximately $t=0.064$.
\end{example}

\begin{example} \label{exm:example2}
    Let us now consider another numerical example on $\Omega=(0,50)^2$ for
    \[u_0 = 5 \left(u_{0}^{(1)} + u_{0}^{(2)} + u_{0}^{(3)} + u_{0}^{(4)}\right),\]
    where 
    \[u_{0}^{(i)}(x,y) = \left\{
        \begin{array}{ll}
            \exp \left(-\dfrac{64}{16-\left(x-x^{(i)}\right)^2-\left(y-y^{(i)}\right)^2}\right), 
             & \left(x-x^{(i)}\right)^2+\left(y-y^{(i)}\right)^2 < 16,  \\ 
            0, & \textrm{else},  
        \end{array}
    \right.\]
    $i=1,2,3,4$. We choose $\left(x^{(1)},y^{(1)}\right)=(15,15)$,
    $\left(x^{(2)},y^{(2)}\right)=(35,15)$, $\left(x^{(3)},y^{(3)}\right)=(15,35)$,
    $\left(x^{(4)},y^{(4)}\right)=(35,35)$,
    \[p(x,y)=2+\left(\frac{x}{25}-1\right)^2+\left(\frac{y}{25}-1\right)^2, \quad 
    k(t)=\frac{1}{400} \left(1+\frac{2}{\pi} \arctan t\right),\]
    and set $N_x = N_y = 500$, $\alpha = -0.05$, $s=0.9$, $\Delta t = 0.5$. 
    Figure~\ref{fig:example2} displays surface plots of the initial data and 
    the numerical solutions at $t=5, 50$ and $500$.
    Figure \ref{fig:example2-px} illustrates the chosen $p(x,y)$, and
    the evolution of $F_1(t)$ on time is shown in Figure \ref{fig:example2-f1t}.
    Over the time interval $[0,500]$, the $L^2$-norm of the numerical solution 
    decreases monotonically and approaches zero, exhibiting smoothing and decaying behavior.
\end{example}

We now proceed to show how \eqref{eqn:main} can be applied to image sharpening. 
In the subsequent experiments, we consistently set $\Delta t=5 \times 10^{-4}$, 
$\alpha=-0.75$, $s=0.9$, $k(t)=0.1 \times 5^{9t}$ and 
\[p(x,y)=\min \left\{3.5, 3 + 0.1 |\nabla u_0(x,y)|^{0.25}\right\}.\]
Let us revisit the analysis of the model presented in the introduction.
Under the above setting, the function $p(x,y)$ is designed to take larger 
values at locations where the initial gradient $|\nabla u_0|$ is large. 
When $t$ is not too large so that $k(t)<1$, locations with larger 
$|\nabla u_0|$ correspond to smaller thresholds $\mathfrak{T}$, favoring 
the occurrence of backward diffusion; in contrast, locations with smaller 
$|\nabla u_0|$ yield larger thresholds $\mathfrak{T}$, making forward 
diffusion more likely. Moreover, since $k(t)$ is designed to be increasing 
in time, the threshold $\mathfrak{T}$ decreases with time at each spatial 
location, which encourages noise suppression in the very early stages 
of the evolution and promotes rapid edge enhancement thereafter, ultimately 
leading to better visual quality.

\begin{figure}[htbp] 
    \centering
    \subfigure[Original]
    {
        \centering
        \includegraphics[width=.3\textwidth]{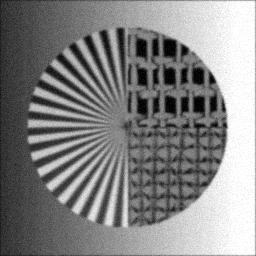} \label{fig:hybrid-ori} 
    }    
    \subfigure[Linear backward diffusion]
    {
        \centering
        \includegraphics[width=.3\textwidth]{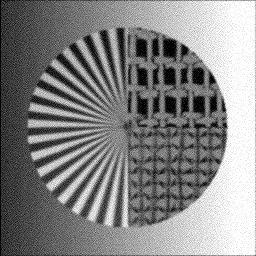} \label{fig:hybrid-bh}
    }

    \subfigure[Shock filter]{
        \centering
        \includegraphics[width=.3\textwidth]{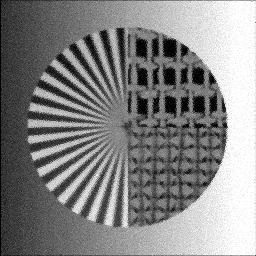}
    }
    \subfigure[Ours]{
        \centering
        \includegraphics[width=.3\textwidth]{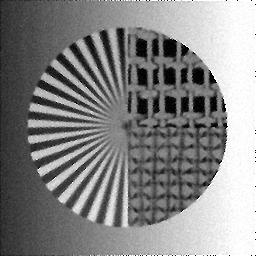}
    }
    \caption{(a) Blurry and noisy synthetic image. (b) Sharpening result using 
    linear backward diffusion equation \eqref{eqn:bh-bwd} with 
    $\varepsilon=0.001$ and $t=0.2$. 
    (c) Sharpening result using shock filter with $t=0.5$. 
    (d) Sharpening result using the proposed model \eqref{eqn:main} with $t=0.025$.} 
    \label{fig:hybrid}
\end{figure}

\begin{figure}[htbp] 
    \centering
    \subfigure[Original]
    {
        \centering
        \includegraphics[width=.3\textwidth]{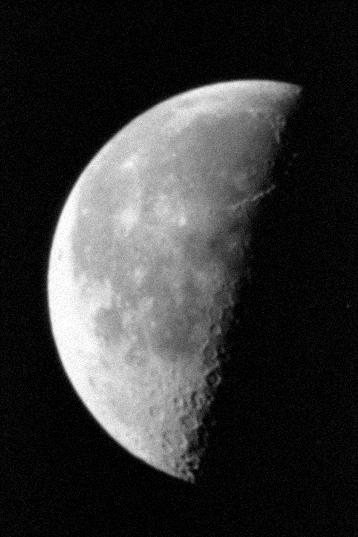}
    }    
    \subfigure[Linear backward diffusion]
    {
        \centering
        \includegraphics[width=.3\textwidth]{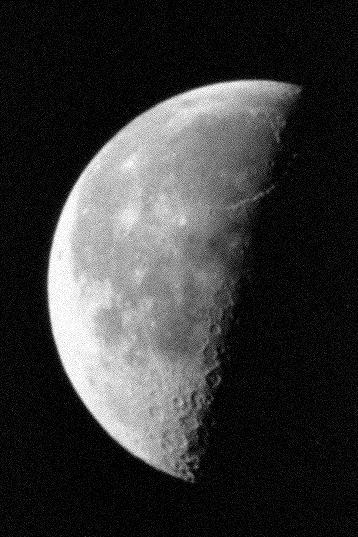}
    }

    \subfigure[Shock filter]{
        \centering
        \includegraphics[width=.3\textwidth]{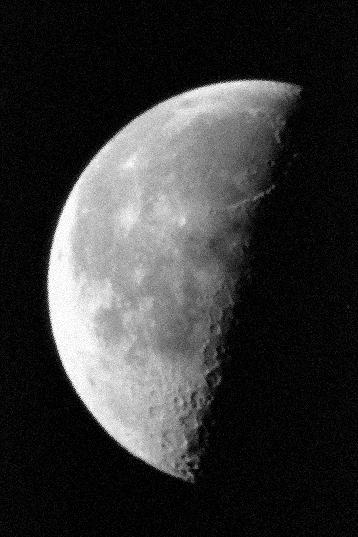}
    }
    \subfigure[Ours]{
        \centering
        \includegraphics[width=.3\textwidth]{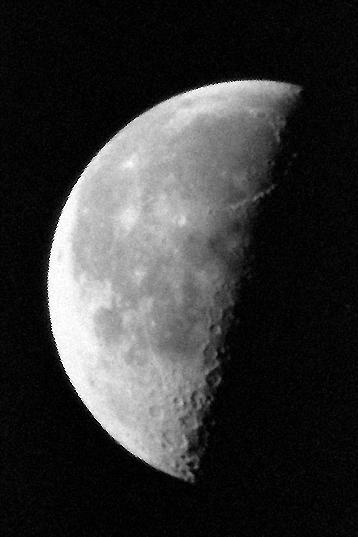}
    }
    \caption{(a) Blurry and noisy moon image. (b) Sharpening result using 
    linear backward diffusion equation \eqref{eqn:bh-bwd} with 
    $\varepsilon=0.001$ and $t=0.2$. 
    (c) Sharpening result using shock filter with $t=0.5$. 
    (d) Sharpening result using the proposed model \eqref{eqn:main} with $t=0.025$.} 
    \label{fig:moon}
\end{figure}

\begin{figure}[htbp] 
    \centering
    \subfigure[Original]
    {
        \centering
        \includegraphics[width=.25\textwidth]{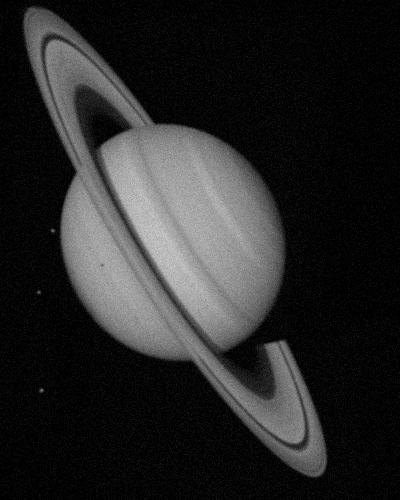} \label{fig:saturn-ori}
    }    
    \subfigure[Sharpened]
    {
        \centering
        \includegraphics[width=.25\textwidth]{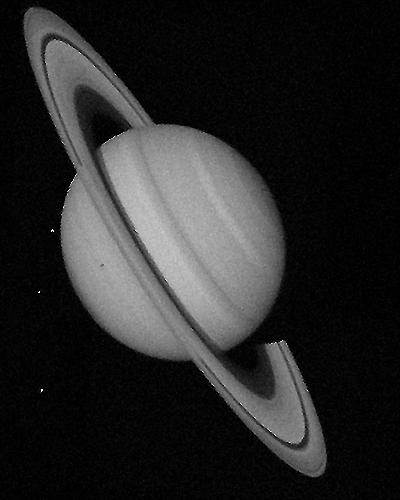}
    }
    \subfigure[Contrast-enhanced]{
        \centering
        \includegraphics[width=.25\textwidth]{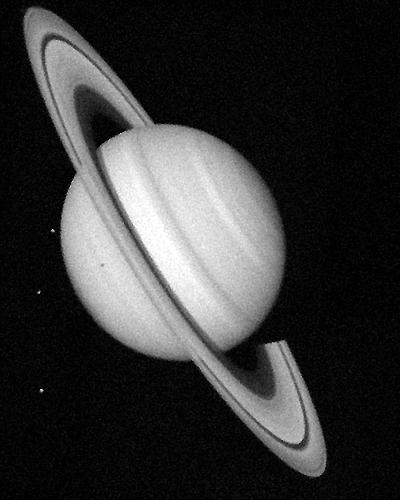} \label{fig:saturn-enhance}
    }
    \caption{(a) Blurry and noisy Saturn image. 
    (b) Sharpening result using model \eqref{eqn:main} with $t=0.05$.
    (c) Contrast enhancement result using model \eqref{eqn:main-source} 
    with $t=0.03$ and $\lambda=10$.} \label{fig:saturn}
\end{figure}

\begin{figure}[htbp] 
    \centering
    \subfigure[Original]
    {
        \centering
        \includegraphics[width=.25\textwidth]{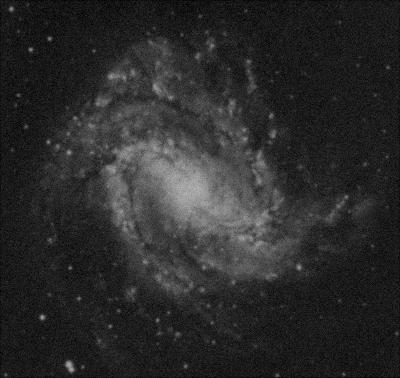} \label{fig:galaxy-ori}
    }    
    \subfigure[Sharpened]
    {
        \centering
        \includegraphics[width=.25\textwidth]{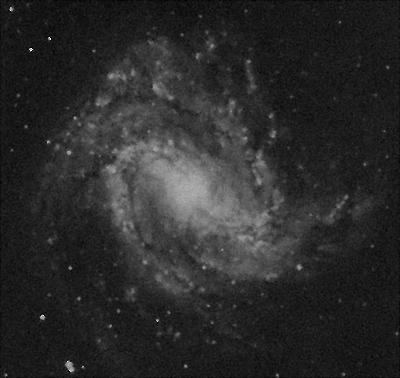}
    }
    \subfigure[Contrast-enhanced]{
        \centering
        \includegraphics[width=.25\textwidth]{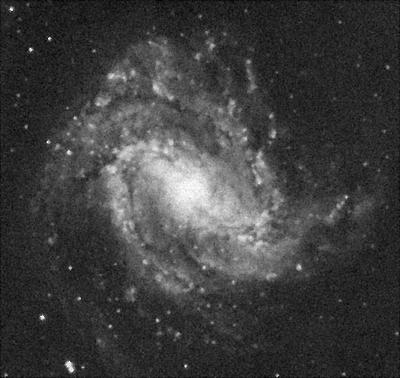} \label{fig:galaxy-enhance}
    }
    \caption{(a) Blurry and noisy galaxy image. 
    (b) Sharpening result using model \eqref{eqn:main} with $t=0.04$.
    (c) Contrast enhancement result using model \eqref{eqn:main-source} 
    with $t=0.025$ and $\lambda=15$.} \label{fig:galaxy}
\end{figure}

Figure \ref{fig:hybrid} presents the edge enhancement results on a synthetic image. 
Alongside the proposed model \eqref{eqn:main}, sharpening results obtained 
using a linear backward diffusion equation with a fourth-order regularization term: 
\begin{equation}
u_t + \varepsilon \Delta^2 u = -\Delta u \label{eqn:bh-bwd}
\end{equation}
and the shock filter 
\[u_t  = - \frac{|\nabla u|}{1+|\Delta u|}\Delta u\]
are also included for comparison. As shown in Fig. \ref{fig:hybrid-bh}, the sharpening 
result obtained using the linear backward diffusion significantly amplifies the noise 
in the original image, which is a common drawback of such type of methods. 
Both the shock filter and the proposed model \eqref{eqn:main} yield visually better results;
however, the edges in the result produced by \eqref{eqn:main} appear sharper, and 
the overall noise level is lower. In addition to performing edge enhancement, the proposed 
model also largely preserves the texture details inside the circular region and the smooth 
ramp features outside. Figure \ref{fig:moon} shows the results of applying the 
aforementioned sharpening methods to a moon test image. It can be observed that the 
proposed model achieves a visually pleasing result: the edges of the moon become sharper, 
surface details such as craters are well-preserved, and the noise in the original image 
has minimal impact on the sharpening result.

At the end of this section, we point out that the proposed model can be easily modified 
into a contrast enhancement model by introducing a linear source term $\lambda u$ on the 
right-hand side of \eqref{eqn:main}, leading to the following form:
\begin{equation}
    u_t + \mathcal{L}_{\alpha} u =
    \operatorname{div}\left((1-k(t)|\nabla u|^{p(x)-2})\nabla u\right) + \lambda u, 
    \label{eqn:main-source}
\end{equation}
This modification is similar to the approaches in \cite{YANG202229,DONG20156060}
and enables simultaneously enhancing the edges and overall contrast of the image. 
Figures \ref{fig:saturn} and \ref{fig:galaxy} show the results of applying 
\eqref{eqn:main} and \eqref{eqn:main-source} to enhance the Saturn and galaxy 
test images, respectively. Figures \ref{fig:saturn-ori} and \ref{fig:galaxy-ori} 
are original images that have blurry edges and low contrast. 
After enhancement by \eqref{eqn:main-source}, the resulting images shown in 
Figs. \ref{fig:saturn-enhance} and \ref{fig:galaxy-enhance} exhibit sharper edges 
and improved overall contrast, achieving desirable visual quality.

\end{document}